\documentclass[11pt]{amsart}
\usepackage{comment}
\usepackage{amssymb,amsmath,amsthm,mathrsfs,amstext,amssymb}
\usepackage{tabularx}
\usepackage{array}
\usepackage[square,sort,comma,numbers]{natbib}
\usepackage{url} 
\usepackage{hyperref}
\usepackage{accents}
\usepackage[shortlabels]{enumitem}

\newtheorem{theorem}{Theorem}
\numberwithin{theorem}{section}
\newtheorem*{theorem*}{Theorem}
\newtheorem{lemma}[theorem]{Lemma}
\newtheorem*{lemma*}{Lemma}
\newtheorem{corollary}[theorem]{Corollary}
\newtheorem*{corollary*}{Corollary}
\newtheorem{proposition}[theorem]{Proposition}
\newtheorem*{proposition*}{Proposition}

\theoremstyle{definition}

\newtheorem{definition}[theorem]{Definition}
\newtheorem*{definition*}{Definition}
\newtheorem{remark}[theorem]{Remark}
\newtheorem*{remark*}{Remark}
\newtheorem{question}[theorem]{Question}
\newtheorem*{question*}{Question}

\newcommand{\A}{{\mathcal{A}}}
\newcommand{\B}{{\mathcal{B}}}

\newcommand{\D}{{\mathcal{D}}}

\newcommand{\F}{{\mathcal{F}}}
\newcommand{\G}{{\mathcal{G}}}

\renewcommand{\P}{{\mathcal{P}}}

\newcommand{\R}{{\mathcal{R}}}
\renewcommand{\S}{{\mathcal{S}}}
\newcommand{\T}{{\mathcal{T}}}

\newcommand{\NN}{{\mathbb{N}}}

\newcommand{\PP}{{\mathbb{P}}}

\renewcommand{\SS}{{\mathbb{S}}}

\newcommand{\ZZ}{{\mathbb{Z}}}

\renewcommand{\a}{{\mathfrak{a}}}
\renewcommand{\b}{{\mathfrak{b}}}

\renewcommand{\d}{{\mathfrak{d}}}

\renewcommand{\i}{{\mathfrak{i}}}

\newcommand{\p}{{\mathfrak{p}}}

\renewcommand{\r}{{\mathfrak{r}}}
\newcommand{\s}{{\mathfrak{s}}}
\renewcommand{\t}{{\mathfrak{t}}}
\renewcommand{\u}{{\mathfrak{u}}}

\DeclareMathOperator{\dom}{dom}
\DeclareMathOperator{\ran}{ran}
\DeclareMathOperator{\restr}{\upharpoonright}
\DeclareMathOperator{\concat}{{^\smallfrown}}
\DeclareMathOperator{\partialto}{{\overset{\text{part}}{\to}}}

\newcommand{\simpleset}[1]{{\{{#1}\}}}
\newcommand{\simpleseq}[1]{{\langle{#1}\rangle}}
\newcommand{\set}[2]{{\{ {#1} \mid {#2} \}}}
\newcommand{\seq}[2]{{\langle {#1} \mid {#2} \rangle}}

\DeclareMathOperator{\extends}{{\mathbin{\leq}}}

\DeclareMathOperator{\forces}{{ \, \Vdash \, }}
\DeclareMathOperator{\notforces}{{ \, \not\Vdash \, }}

\newcommand{\gen}{{\dot{G}}}

\newcommand{\finsubset}[1]{{[#1]^{<\omega}}}
\newcommand{\infsubset}[1]{{[#1]^{\omega}}}

\DeclareMathOperator{\aE}{\a_{\text{\normalfont e}}}
\DeclareMathOperator{\aG}{\a_{\text{\normalfont g}}}
\DeclareMathOperator{\aT}{\a_{\text{\normalfont T}}}
\DeclareMathOperator{\bairespace}{{^\omega \omega}}
\DeclareMathOperator{\cantorspace}{{^\omega 2}}
\DeclareMathOperator{\cantorseqspace}{{^\omega}(\cantorspace)}
\DeclareMathOperator{\finbairespace}{{^{<\omega} \omega}}
\DeclareMathOperator{\fincantorspace}{{^{<\omega} 2}}
\DeclareMathOperator{\finseqcantorspace}{{^{<\omega} (\fincantorspace)}}
\DeclareMathOperator{\code}{code}
\DeclareMathOperator{\eoi}{EoI}

\DeclareMathOperator{\dominatedby}{{<^*}}

\DeclareMathOperator{\ndominatedby}{{\not<^*}}

\DeclareMathOperator{\sucspl}{succspl}
\DeclareMathOperator{\subsequence}{\trianglelefteq}
\DeclareMathOperator{\spl}{spl}
\DeclareMathOperator{\stem}{stem}

\DeclareMathOperator{\id}{id}
\DeclareMathOperator{\fin}{fin}
\DeclareMathOperator{\Sinf}{{S_\infty}}
\DeclareMathOperator{\Sinffin}{{S_\infty^{\fin}}}
\DeclareMathOperator{\Sinfplus}{{S_\infty^{+}}}
\DeclareMathOperator{\fix}{fix}
\DeclareMathOperator{\cofin}{cofin}
\DeclareMathOperator{\defined}{{\downarrow}}
\DeclareMathOperator{\undefined}{{\uparrow}}

\DeclareMathOperator{\sign}{sign}

\title{Universally Sacks-indestructible combinatorial families of reals}

\author{V. Fischer}
\address{Institute of Mathematics, University of Vienna, Kolingasse 14-16, 1090 Vienna, Austria}
\email{vera.fischer@univie.ac.at}

\author{L. Schembecker}
\address{Institute of Mathematics, University of Vienna, Kolingasse 14-16, 1090 Vienna, Austria}
\email{lukas.schembecker@univie.ac.at}

\thanks{\emph{Acknowledgements.}: The authors would like to thank the Austrian Science Fund (FWF) for the generous support through START Grant Y1012-N35.}



\begin{document}
		
	\begin{abstract}
		We introduce the notion of an arithmetical type of combinatorial family of reals, which serves to generalize different types of families such as mad families, maximal cofinitary groups, ultrafilter bases, splitting families and other similar types of families commonly studied in combinatorial set theory.
		
		We then prove that every combinatorial family of reals of arithmetical type which is indestructible by the product of Sacks forcing $\SS^{\aleph_0}$ is in fact universally Sacks-indestructible, i.e.\ it is indestructible by any countably supported iteration or product of Sacks-forcing of any length.
		Further, under $\sf{CH}$ we present a unified construction of universally Sacks-indestructible families for various arithmetical types of families.
		In particular we prove the existence of a universally Sacks-indestructible maximal cofinitary group under $\sf{CH}$.
	\end{abstract}
	
	\maketitle
	
	\section{Introduction}\label{SEC_Intro}
	
	In combinatorial set theory there is a vast amount of different combinatorial families of reals to be studied, e.g.\ mad families, maximal cofinitary groups, maximal independent families, ultrafilter bases, unbounded families, splitting families, maximal eventually different families and many others.
	We refer to these as different types of combinatorial families of reals.
	Assume we fixed some type of combinatorial family and let $\F$ be a family of that type.
	Then, in any forcing extension we might introduce new reals, which witness that $\F$ is not a family of our fixed type any more.
	For example for a mad family $\F$ we might add a new real which has finite intersection with every element of $\F$ so that $\F$ is not maximal in the forcing extension any more.
	We call such reals intruders for $\F$ and for a forcing $\PP$ we say that $\PP$ preserves $\F$ iff there are no intruders for $\F$ after forcing with $\PP$.
	We also say $\F$ is $\PP$-indestructible.
	
	Note that the notion of an intruder for $\F$ heavily depends on the type of family at hand.
	In many examples an intruder is a real that witnesses non-maximality, e.g.\ for mad families or independent families.
	However, in other contexts an intruder can also have various other interpretations, for example:
	
	\begin{enumerate}[$\circ$]
		\item unbounded family $\F$ $\longrightarrow$ a real dominating $\F$,
		\item splitting family $\F$ $\longrightarrow$ a set not split by $\F$,
		\item ultrafilter basis $\F$ $\longrightarrow$ a set $A$ with both $A$ and $A^c$ not in the filter generated by $\F$.
	\end{enumerate}
	
	One interesting research direction is the study of which forcings may preserve which types of combinatorial families.
	Given the vast amount of possible combinations, there are various constructions of indestructible families, alternative characterizations of the indestructibility of such families and implications between different types of forcing indestructibilities.
	First, we give a small non-exhaustive overview:
	
	Maximal almost disjoint (mad) families are the most well studied case.
	In \cite{Kunen} Kunen constructed a Cohen-indestructible mad family under {\sf{CH}} and Hru\v{s}\'{a}k \cite{Hrusak} and Kurili\'{c} \cite{Kurilic} independently proved combinatorial characterizations of Cohen-indestructibility of mad families.
	These ideas have also been expanded to other types of forcings, such as Sacks, Miller, Laver and random forcing in \cite{Hrusak}\cite{SacksMAD}.
	Moreover, in \cite{SacksMAD} Brendle and Yatabe proved implications between these different types of forcing indestructibilities and also considered characterizations of iterated Sacks-indestructibility of mad families.
	
	For other types of families usually Cohen and Sacks-indestructibility are the most well-studied cases.
	In \cite{CohenMCG} Fischer, Schrittesser and Törnquist constructed a Cohen-indestructible maximal cofinitary group.
	The construction may be adapted to also obtain a Cohen-indestructible maximal eventually different (med) family.
	For a construction of a Sacks-indestructible med family, see~\cite{FischerSchrittesser}.
	In \cite{Switzer} Fischer and Switzer adapted the notion of tightness to med families and proved that it implies Cohen-indestructibility.
	For independent families Shelah \cite{ShelahUlessI} implicitly proved the existence of a Sacks-indestructible maximal independent family, also see \cite{SacksMIF1} and \cite{SacksMIF2} for an explicit construction.
	In \cite{New87} Newelski forced the existence of a product Sacks-indestructible partition of Baire space into $F_\sigma$-meager~sets.
	Further, Sacks-indestructibility of ultrafilter bases is closely related to the Halpern-Lauchli theorem \cite{HalpernLauchli}.
	Laver generalized Halpern and Lauchli's results to prove that every selective ultrafilter is product Sacks-indestructible \cite{Laver}; also see \cite{HL} for an analysis of Sacks-indestructibility of ultrafilters and reaping families.
	Of particular interest for this paper is the construction of a maximal eventually different family which is indestructible under any product or iteration of Sacks forcing of any length by Fischer and Schrittesser in \cite{FischerSchrittesser}.
	Since this property is crucial for this paper we define:
	
	\begin{definition*}
		A family $\F$ is called universally Sacks-indestructible iff $\F$ is indestructible under any product or iteration of Sacks forcing of any length.
	\end{definition*}
	
	In order to obtain such a med family Fischer and Schrittesser \cite{FischerSchrittesser} constructed a $\SS^{\aleph_0}$-indestructible med family and proved that this family is in fact universally Sacks-indestructible, where $\SS^{\aleph_0}$ is the product of Sacks forcing.
	In \cite{FischerSchembecker} the authors proved that a similar construction may also be used to construct a universally Sacks-indestructible partition of Baire space into compact sets.
	In this paper we will generalize these findings to various other types of combinatorial families.
	To this end, we introduce the notion of an arithmetical type of combinatorial family of reals and formally define the notion of an intruder (cf.\ Definition~\ref{DEF_type}), which generalizes all the different types of combinatorial families of reals mentioned above:
	
	\begin{definition*}
		An arithmetical type $\t$ (of combinatorial family of reals) is a pair of sequences $\t = ((\psi_n)_{n < \omega}, (\chi_n)_{n < \omega})$ such that both $\psi_n(w_0, w_1, \dots, w_n)$ and $\chi_n(v,w_1, \dots, w_n)$ are arithmetical formulas in $n + 1$ real parameters.
		The domain of the type $\t$ is the set
		$$
		\dom(\t) := \set{\F \subseteq \P(\bairespace)}{\forall n < \omega \ \forall \simpleset{f_0, \dots, f_n} \in [\F]^{n + 1} \text{ we have } \psi_n(f_0, \dots, f_n)}
		$$
		If $\F \in \dom(\t)$ we say $\F$ is of type $\t$.
		Now, let $\F$ be of type $\t$.
		If a real $g$ satisfies
		$$
		\forall n < \omega \forall \simpleset{f_1, \dots, f_n} \in [\F]^n \ \chi_n(g, f_1, \dots, f_n),
		$$
		then we call $g$ an intruder for $\F$.
	\end{definition*}
	
	Thus, in the notion of an arithmetical type we essentially require that what constitutes a suitable family and an intruder is definable by a sequence of arithmetical formulas in the above sense.
	We then prove the following Theorem~\ref{THM_MainTheorem}:
	
	\begin{theorem*}
		Assume that $\t$ is an arithmetical type and $\F$ is a $\SS^{\aleph_0}$-indestructible family of type $\t$.
		Then $\F$ is universally Sacks-indestructible.
	\end{theorem*}
	
	For example we immediately obtain the following new facts (see Corollaries~\ref{COR_MAD}~and~\ref{COR_UANDI}):
	
	\begin{corollary*}
		Every $\SS^{\aleph_0}$-indestructible mad family is universally Sacks-indestructible.
	\end{corollary*}
	
	\begin{corollary*}
		Every $\SS^{\aleph_0}$-indestructible independent family and every $\SS^{\aleph_0}$-indestructible ultrafilter is universally Sacks-indestructible.
	\end{corollary*}

	In particular, not only the med family constructed by Fischer and Schrittesser in \cite{FischerSchrittesser} is universally Sacks-indestructible, but in fact every $\SS^{\aleph_0}$-indestructible family already is.
	We also generalize the constructive part of their proof to obtain universally Sacks-indestructible families of various types under $\sf{CH}$.
	However, in order to present a unified construction, we require the following additional property (see Definition~\ref{DEF_EliminatingIntrudersLemma}):
	
	\begin{definition*}
		Let $\t$ be an arithmetical type.
		We say that $\t$ satisfies elimination of intruders and write $\eoi(\t)$ holds iff the following property is satisfied:
		If $\F$ is a countable family of type $\t$, $p \in \SS^{\aleph_0}$ and $\dot{g}$ is a name for a real such that
		$$
		p \forces \dot{g} \text{ is an intruder for } \F.
		$$
		Then there is $q \extends p$ and a real $f$ such that $\F \cup \simpleset{f}$ is of type $\t$ and
		\[
		q \forces \dot{g} \text{ is not an intruder for } \F \cup \simpleset{f}.
		\]
	\end{definition*}

	Now, if $\eoi(\t)$ is satisfied we prove a unified construction of a universally Sacks-indestructible witness under $\sf{CH}$ in Theorem~\ref{THM_SacksUniversal}:
	
	\begin{theorem*}
		Assume $\sf{CH}$ and $\eoi(\t)$ holds.
		Then there is a universally Sacks-indestructible family of type $\t$.
	\end{theorem*}

	In Lemma~\ref{LEM_MCG} we prove that elimination of intruders indeed holds for maximal cofinitary groups, so that we may apply the previous theorem to obtain the following new result (see Corollary~\ref{COR_MCG}):
	
	\begin{corollary*}
		Under $\sf{CH}$ there is a universally Sacks-indestructible maximal cofinitary group.
	\end{corollary*}

	Finally, since the definition of an arithmetical type requires us to work with arithmetical formulas, we prove the following technical Lemma~\ref{LEM_ForcingToAbsolute} which is an interesting result on its own.
	Essentially, it allows us for any condition $p \in \SS^{\aleph_0}$ to translate the statement ``$p$ forces an arithmetical property of the generic sequence $s_\gen$" into an equivalent $\Pi^1_3$-statement (see Lemma~\ref{LEM_ForcingToAbsolute}):
	
	\begin{lemma*}
		Let $\chi(v_1, \dots, v_k, w_1, \dots, w_l)$ be an arithmetical formula in $k + l$ real parameters.
		Further, let $p \in \SS^{\aleph_0}$, $f_1, \dots, f_l$ be reals and $g_1, \dots, g_k$ be codes.
		Then the following are equivalent:
		\begin{enumerate}
			\item $p \forces \chi(g_1^*(s_\gen), \dots, g_k^*(s_\gen), f_1, \dots, f_l)$,
			\item $\forall q \extends p \ \exists r \extends q \ \forall x \in [r] \ \chi(g_1^*(x), \dots, g_k^*(x), f_1, \dots, f_l)$.
		\end{enumerate}
	\end{lemma*}

	Here, $s_\gen \in \cantorseqspace$ is the name for the generic sequence of Sacks-reals, the codes $g_i$ are interpreted as continuous functions $g_i^*:\cantorseqspace \to \bairespace$ and an arithmetical formula $\chi$ is a first-order formula with possibly real parameters, so that $\chi$ only contains integer quantifiers.
	
	This paper is structured as follows:
	In the second section we revisit all necessary preliminaries such as all important notions for Sacks forcing and its fusion, followed by a similar discussion for countably supported product/iteration of Sacks forcing.
	Furthermore, we will want to apply a nice version of continuous reading of names for countably supported products/iterations of Sacks forcing developed by Fischer and Schrittesser in \cite{FischerSchrittesser}, so in order to state their result we also go over some technicalities concerning coding of continuous functions.
	
	In the third section we prove the technical Lemma~\ref{LEM_ForcingToAbsolute} just mentioned and the implication from $\SS^{\aleph_0}$-indestructibility to universal Sacks-indestructibility.
	Finally, in the fourth section, we prove the existence of a universally Sacks-indestructible witness under $\sf{CH}$ given that elimination of intruders holds (Theorem~\ref{THM_SacksUniversal}).
	We then show that various different types of combinatorial families of reals fit in our framework of arithmetical types.
	Section \ref{SECSUB_MAD} covers mad families, Section \ref{SECSUB_MED} med families and Section \ref{SECSUB_ADFS} partitions of Baire space into compact sets.
	Section \ref{SECSUB_MCG} and \ref{SECSUB_ELIMINATING_MCG} handle maximal cofinitary groups. 
	There, we also reintroduce the notion of nice words as there are some inaccuracies in the literature.
	Finally, in Section \ref{SECSUB_UANDI} we cover independent families and ultrafilter bases and finish with other types of families such as unbounded, dominating, splitting and reaping families in the last Section \ref{SECSUB_OTHER}.
	For mad families, med families and maximal cofinitary groups we also provide proofs for elimination of intruders in their respective sections.
	
	\section{Preliminaries}\label{SEC_Prelim}
	
	First, we consider the basic notions and definitions used throughout this paper.
	We start with iterations and products of Sacks forcing and their fusion.
	The following lemmata are well-known, for a more detailed presentation see \cite{Kanamori} for example.
	
	\begin{definition}
		Let $T \subseteq \fincantorspace$ be a tree, i.e.\ $T$ is closed under initial subsequences.
		\begin{enumerate}
			\item For $s, t \in \fincantorspace$ we write $s \subsequence t$ iff $s$ is an initial subsequence of $t$.
			\item Let $s \in T$ then $T_s := \set{t \in T}{s \subsequence t \text{ or } t \subsequence s}$.
			\item $\spl(T) := \set{s \in T}{s \concat 0 \in T \text{ and } s \concat 1 \in T}$ is the set of all splitting nodes of $T$.
			\item $T$ is perfect iff for all $s \in T$ there is $t \in \spl(T)$ such that $s \subsequence t$.
			\item $\SS := \set{T \subseteq \fincantorspace}{T\text{ is a perfect tree}}$ ordered by inclusion is Sacks forcing.
		\end{enumerate}
	\end{definition}
	
	\begin{definition}
		Let $T \in \SS$.
		We define the fusion ordering for Sacks forcing as follows:
		\begin{enumerate}
			\item Let $s \in T$ then $\sucspl_T(s)$ is the unique minimal splitting node in $T$ extending $s$.
			\item $\stem(T) := \sucspl_T(\emptyset)$.
			\item $\spl_0 := \simpleset{\stem(T)}$ and if $\spl_n(T)$ is defined for $n < \omega$ we set
			$$
				\spl_{n + 1} := \set{\sucspl_T(s \concat i)}{s \in \spl_n(T), i \in 2}.
			$$
			$\spl_n(T)$ is called the $n$-th splitting level of $T$.
			\item Let $n < \omega$ and $S, T \in \SS$.
			We write $S \extends_n T$ iff $S \subseteq T$ and $\spl_n(S) = \spl_n(T)$.
		\end{enumerate}
	\end{definition}
	
	\begin{lemma}
		Let $\seq{T_n \in \SS}{n < \omega}$ be a sequence of trees such that $T_{n + 1} \extends_n T_n$ for all $n < \omega$.
		Then $T := \bigcap_{n < \omega}T_n \in \SS$ and $T \extends_n T_n$ for all $n < \omega$.
	\end{lemma}

	We call the sequence $\seq{T_n \in \SS}{n < \omega}$ above a fusion sequence in $\SS$ and the element $T \in \SS$ its fusion.
	
	\begin{definition}
		Let $\lambda$ be a cardinal.
		$\SS^\lambda$ is the countable support product of Sacks forcing of size $\lambda$. Moreover,
		\begin{enumerate}
			\item for $A \subseteq \SS^\lambda$ let $\bigcap A$ be the function with $\dom(\bigcap A) := \bigcup_{p \in A} \dom(p)$ and for all $\alpha < \lambda$ we have $(\bigcap A)(\alpha) := \bigcap_{p \in A} p(\alpha)$.
			Notice that we do not necessarily have $\bigcap A \in \SS^\lambda$.
			\item Let $n < \omega$, $p, q \in \SS^\lambda$ and $F \in \finsubset{\dom(q)}$.
			Write $p \extends_{F, n} q$ iff $p \extends q$ and $p(\alpha) \extends_n q(\alpha)$ for all $\alpha \in F$.
			For $\lambda = \aleph_0$ we assume every condition has full support and write $\leq_{n}$ for $\leq_{n, n}$.
		\end{enumerate}
	\end{definition}
	
	\begin{lemma}
		Let $\seq{p_n \in \SS^\lambda}{n < \omega}$ and $\seq{F_n \in \finsubset{\dom(p_n)}}{n < \omega}$ be sequences such that
		\begin{enumerate}
			\item $p_{n + 1} \extends_{F_n, n} p_n$ for all $n < \omega$.
			\item $F_{n} \subseteq F_{n + 1}$ for all $n < \omega$ and $\bigcup_{n < \omega} F_n = \bigcup_{n < \omega} \dom(p_n)$.
		\end{enumerate}
		Then $p := \bigcap_{n < \omega} p_n \in \SS^\lambda$ and $p \extends_{F_n, n} p_n$ for all $n < \omega$.
	\end{lemma}

	Again, we call the sequence $\seq{p_n \in \SS^\lambda}{n < \omega}$ above a fusion sequence in $\SS^{\lambda}$ for $\seq{F_n}{n < \omega}$ and the element $p \in \SS^\lambda$ its fusion.
	In order to construct such fusion sequences we use the notion of suitable functions:
	
	\begin{definition}
		Let $p \in \SS^\lambda$, $F \in \finsubset{\dom(p)}$, $n < \omega$ and $\sigma:F \to V$ be a suitable function for $p$, $F$ and $n$, i.e.\ $\sigma(\alpha) \in \spl_n(p(\alpha)) \concat 2$ for all $\alpha \in F$.
		Then we define $p \restr \sigma \in \SS^\lambda$ by
		$$
		(p \restr \sigma) (\alpha) := 
		\begin{cases}
			p(\alpha)_{\sigma(\alpha)} & \text{if } \alpha \in F, \\
			p(\alpha) & \text{otherwise}.
		\end{cases}
		$$
		Notice that for fixed $p \in \SS^\lambda$, $n < \omega$ and $F \in \finsubset{\dom(p)}$ there are only finitely many $\sigma$ which are suitable for $p$, $F$ and $n$.
		Also, if $q \extends_{F, n} p$, then $q$ and $p$ have the same suitable functions for $F$ and $n$.
		Furthermore, the set
		$$
		\set{p \restr \sigma}{\sigma:F \to V \text{ is a suitable function for } p, F \text{ and } n}
		$$
		is a maximal antichain below $p$.
		Again, if $\lambda = \aleph_0$ we just say $\sigma$ is suitable for $p$ and $n$ in case that $\sigma$ is suitable for $p$, $n$ and $n$.
	\end{definition}

	For most fusion arguments in the subsequent sections we will need the following easy lemma:
	
	\begin{lemma}\label{LEM_ExtendDenseOpen}
		Let $p \in \SS^\lambda$, $F \in\finsubset{\dom(p)}$, $n < \omega$ and $D \subseteq \SS^{\lambda}$ be open dense below $p$.
		Then there is $q \leq_{F, n} p$ such that $q \restr \sigma \in D$ for all $\sigma$ suitable for $p$, $F$ and $n$.
	\end{lemma}

	\begin{proof}
		Let $\seq{\sigma_i}{i < N}$ enumerate all suitable functions for $p$, $F$ and $n$ and set $q_0 := p$. We will define a $\extends_{F, n}$-decreasing sequence $\seq{q_i}{i \leq N}$ so that all of the $q_i$ have the same suitable functions as $p$ for $F$ and $n$.
		Assume $i < N$ and $q_i$ is defined. Choose $r_i \extends q_i \restr \sigma_i$ in $D$ and define
		$$
		q_{i + 1}(\alpha) :=
		\begin{cases}
			r_i(\alpha) \cup \bigcup \set{q_i(\alpha)_s}{s \in \spl_n(q_i(\alpha)) \concat 2 \text{ and } s \neq \sigma(\alpha)} & \text{if } \alpha \in F \\
			r_i(\alpha) & \text{otherwise}
		\end{cases}
		$$
		Clearly, $q_{i + 1} \extends_{F, n} q_i$ and $q_{i + 1} \restr \sigma = r_i$. Now, set $q := q_N$ and let $\sigma$ be suitable for $p$, $F$ and $n$. Choose $i < N$ such that $\sigma = \sigma_i$. Then we have $q \restr \sigma \extends q_{i+1} \restr \sigma = r_i \in D$, so $q \restr \sigma \in D$ as $D$ is open.
	\end{proof}
	
	Next, we briefly present a simplified version of the presentation of continuous reading of names for Sacks-forcing in \cite{FischerSchrittesser} suited for our needs, see also \cite{Kechris}.
	First, we consider how to code continuous functions $f^*:\cantorseqspace \to \bairespace$ by monotone and proper functions $f:\finseqcantorspace \to \finbairespace$:
	
	\begin{definition}\
		\begin{enumerate}
			\item For $s, t \in \finseqcantorspace$ write $s \trianglelefteq t$ iff $\dom(s) \leq \dom(t)$ and for all $n \in \dom(s)$, $s(n) \trianglelefteq t(n)$.
			\item A function $f:\finseqcantorspace \to \finbairespace$ is monotone iff for all $s \trianglelefteq t \in \finseqcantorspace$, $f(s) \trianglelefteq f(t)$.
			\item A function $f:\finseqcantorspace \to \finbairespace$ is proper iff for all $x \in {^\omega (^\omega 2)}$:
			$$
				|\dom(f(x \restr n \times n))| \overset{n \to \infty}{\longrightarrow} \infty.
			$$
			\item For a monotone, proper function $f:\finseqcantorspace \to \finbairespace$ define a continuous function
			$$
				f^*:{^\omega (^\omega 2)} \to \bairespace\hbox{ via } f^*(x) := \bigcup_{n < \omega} f(x \restr n \times n).
			$$
			In this case $f$ is called a code for $f^*$.
		\end{enumerate}
	\end{definition}
	
	\begin{remark}
		Conversely, for every continuous function $f^*:{^\omega (^\omega 2)} \to \finbairespace$ there is a code for it.
		In the following a code $f$ will always refer to a monotone and proper function $f: \finseqcantorspace \to \finbairespace$.
	\end{remark}
	
	\begin{remark}
		For all $p, q \in \SS$ there is a natural bijection $\pi:\spl(p) \to \spl(q)$ which for every $n < \omega$ restricts to bijections $\pi \restr \spl_n(p): \spl_n(p) \to \spl_n(q)$ and which preserves the lexicographical ordering.
		We can extend $\pi$ to a monotone and proper function $\pi:p \to q$ in a similar sense as above.
		$\pi$ then codes a homeomorphism $\pi:[p] \to [q]$ which we call the induced homeomorphism (of $p$ and $q$).
		We usually identify both functions $\pi:p \to q$ and $\pi:[p] \to [q]$ with the same letter.
		Note that $\pi$ is indeed a homeomorphism as its inverse is given by the induced homeomorphism from $q$ to $p$.
	\end{remark}

	\begin{lemma}
		Let $p, q, r \in \SS$ such that $r \extends p$ and let $\pi:[p] \to [q]$ be the induced homeomorphism.
		Then there is $s \extends q$ such that $\pi[[r]] = [s].$
	\end{lemma}

	\begin{proof}
		We define
		$$
			s := \set{u \in \fincantorspace}{\exists f \in [r] \ u \subseteq \pi(f)}.
		$$
		We have to show that $s \in \PP$.
		Clearly, $s$ is downwards closed, so let $u \in s$.
		Choose $f \in [r]$ such that $u \subseteq\pi(f)$.
		Since $\pi(f) \in [q]$ we have $u \in q$.
		By definition of the induced map choose $v \in p$ such that $v \subseteq f$ and  $u \subseteq \pi(v)$.
		Then $v \in r$, so choose $w \in r$ with $v \subseteq w$ and $w \in \spl(r)$.
		Then, also $w \in \spl(p)$ and we have $\pi(w) \in \spl(q)$ and $\pi(v) \subseteq \pi(w)$.
		Since $r \in \SS$ for $i \in 2$ we may choose $f_i \in [r]$ such that $w \concat i \subseteq f_i$.
		But $\pi(w) \concat i \subseteq \pi(f_i)$ implies $\pi(w) \concat i \in s$ for $i \in 2$, i.e.\ $\pi(w) \in \spl(s)$.
		Further, $u \subseteq \pi(v) \subseteq \pi(w)$ which proves that $s \in \SS$.
	\end{proof}

 \begin{remark}
 	Notice that $s$ is uniquely determined by the property above.
 	We call $s$ the image of $r$ under $\pi$.
 	In the dual case where $p,q,s \in \SS$ are such that $s \extends q$ and $\pi:[p] \to [q]$ is the induced homeomorphism we say that $r$ is the preimage of $s$ under $\pi$ iff $r$ is the image of $s$ under $\pi^{-1}$.
 	Here, we use that the inverse is given by the induced homeomorphism from $q$ to $p$.
 \end{remark}

	\begin{definition}
		Let $\PP$ be the countably supported iteration of Sacks forcing of length $\lambda \geq \omega$.
		Let $p \in \PP$.
		We may always assume the dense in $\PP$ property that $|\dom(p)| = \omega$ and $0 \in \dom(p)$.
		\begin{enumerate}
			\item A standard enumeration of $\dom(p)$ is a sequence
			$$
				\Sigma = \seq{\sigma_k}{k < \omega},
			$$
			such that $\sigma_0 = 0$ and $\ran(\Sigma) = \dom(p)$.
			\item Let $[p]$ be a $\PP$-name such that
			$$
				p \forces [p] = \seq{x \in {^{\dom(p)}}({^\omega} 2)}{\text{For all } \alpha \in \dom(p) \text{ we have } x(\alpha) \in [p(\alpha)]}.
			$$
			\item Let $\Sigma$ be a standard enumeration of $\dom(p)$.
			For $k < \omega$ let $\dot{e}_k^{p, \Sigma}$ be a $\PP \restr \sigma_k$-name such that
			$$
				p \restr \sigma_k \forces \dot{e}_k^{p, \Sigma} \text{ is the induced homeomorphism between } [p(\sigma_k)] \text{ and } ^\omega 2.
			$$
			Moreover, let $\dot{e}^{p, \Sigma}$ be a $\PP$-name such that
			$$
				p \forces \dot{e}^{p, \Sigma}:[p] \to {^\omega(^\omega 2)} \text{ such that } \dot{e}^{p, \Sigma}(x) = \seq{\dot{e}^{p, \Sigma}_k(x(\sigma_k))}{k < \omega} \text{ for all } x \in [p].
			$$
			\item Given $s \in \SS^{\aleph_0}$ we define the preimage $r$ of $s$ under $\dot{e}^{p, \Sigma}$ as follows.
			Let $r \in \PP$ with $\dom(r) = \dom(p)$, where for $k < \omega$ we have that $r(\sigma_k)$ is a $\PP \restr \sigma_k$-name such that
			$$
				p \restr \sigma_k \forces r(\sigma_k) \text{ is the preimage of } s(k) \text{ under } \dot{e}^{p, \Sigma}_k.
			$$
			In particular we have
			$$
				p \restr \sigma_k \forces r(\sigma_k) \extends p(\sigma_k),
			$$
			so that $r \leq p$.
			Furthermore, $r$ satisfies for every $k < \omega$
			$$
				r \forces s_\gen(\sigma_k) \in [r(\sigma_k)],
			$$
			so that by the previous discussion
			$$
				r \forces \dot{e}^{p, \Sigma}_k(s_\gen(\sigma_k)) \in [s(k)]
			$$
			Thus, we obtain
			$$
				r \forces \dot{e}^{p, \Sigma}(s_\gen \restr \dom(p)) \in [s]
			$$
		\end{enumerate}
	\end{definition}
	
	\begin{remark}
		For the countable support product of Sacks forcing we define the analogous notions. In fact, in this simpler case $[p]$, $\dot{e}^{p, \Sigma}_k$ and $\dot{e}^{p, \Sigma}$ can be defined as ground model objects.
		However, we will still treat them as names, so that we may consider both cases at the same time.
	\end{remark}
	
	\begin{definition}
		Let $\PP$ be the countable support iteration or product of Sacks forcing of any length.
		Let $q \in \PP$ and $\dot{f}$ be a $\PP$-name for a real.
		Let $\Sigma = \seq{\sigma_k}{k < \omega}$ be a standard enumeration of $\dom(q)$ and $f:\finseqcantorspace \to \finbairespace$ be a code for a continuous function $f^*:\cantorseqspace \to \bairespace$ such that
		$$
			q \forces \dot{f} = (f^* \circ \dot{e}^{q, \Sigma})(s_\gen \restr \dom(q)).
		$$
		Then we say $\dot{f}$ is read continuously below $q$ (by $f$ and $\Sigma$).
	\end{definition}
	
	\begin{lemma}[Lemma 4 of \cite{FischerSchrittesser}]\label{LEM_ContinuousReadingOfNames}
		Let $\PP$ be the countable support iteration or product of Sacks forcing of length $\lambda$.
		Suppose $p \in \PP$ and $\dot{f}$ is a $\PP$-name for a real.
		Then there is $q \extends p$ such that $\dot{f}$ is read continuously below $q$.
	\end{lemma}
	
	\begin{remark}
		For any $p \in \PP$ and $\PP$-name $\dot{f}$ for a real it is easy to see that if $\dot{f}$ is read continuously below $p$ then for all $q \extends p$ also $\dot{f}$ is read continuously below $q$.
		Thus, the previous lemma shows that the set
		$$
			\set{q \in \PP}{\dot{f} \text{ is read continuously below } q}
		$$
		is dense open in $\PP$.
	\end{remark}
	
	\section{Main Results}\label{SEC_Main}
	
	Before we dive in into our main results, let us first consider the following application of $\Pi^1_1$-absoluteness.
	Given $p \in \SS$, real parameters $f_1, \dots, f_n$ and a $\Pi^1_1$-formula $\chi(v, w_1, \dots, w_n)$ with $n + 1$ real parameters assume that the following holds
	$$
		\forall q \extends p \ \exists r \extends q \ \forall x \in [r] \ \chi(x, f_1, \dots, f_n).
	$$
	Then we claim that also  $p \forces \chi(s_\gen, f_1, \dots, f_n)$ holds.
	Indeed, let $q \extends p$.
	By assumption choose $r \extends q$ such that
	$$
		\forall x \in [r] \ \chi(x, f_1, \dots, f_n).
	$$
	This is a $\Pi^1_1$-statement so that by $\Pi^1_1$-absoluteness
	$$
		\SS \forces \forall x \in [r] \ \chi(x, f_1, \dots, f_n).
	$$
	But we also we have $r \forces ``s_\gen \in [r]"$ so that
	$$
		r \forces \chi(s_\gen, f_1, \dots, f_n),
	$$
	which proves the statement.
	
	The main goal of this chapter will be to show that if $\chi$ is an arithmetical formula, then we can also prove the converse, namely that $p \forces \chi(s_\gen, f_1, \dots, f_n)$ implies
	$$
		\forall q \extends p \ \exists r \extends q \ \forall x \in [r] \ \chi(x, f_1, \dots, f_n).
	$$
	Even better, we will show that we also have an analogous equivalence for $\SS^{\aleph_0}$ in place of $\SS$.
	Thus, we are able to transfer arithmetical forcing statements of $\SS^{\aleph_0}$ into $\Pi^1_3$-formulas and back, which will be one of the main ingredients for Theorem~\ref{THM_MainTheorem}.
		
	\begin{lemma}\label{LEM_ForcingToAbsolute}
		Let $\chi(v_1, \dots, v_k, w_1, \dots, w_l)$ be an arithmetical formula in $k + l$ real parameters.
		Further, let $p \in \SS^{\aleph_0}$, $f_1, \dots, f_l$ be reals and $g_1, \dots, g_k$ be codes.
		Then the following are equivalent:
		\begin{enumerate}
			\item $p \forces \chi(g_1^*(s_\gen), \dots, g_k^*(s_\gen), f_1, \dots, f_l)$,
			\item $\forall q \extends p \ \exists r \extends q \ \forall x \in [r] \ \chi(g_1^*(x), \dots, g_k^*(x), f_1, \dots, f_l)$.
		\end{enumerate}
	\end{lemma}

	\begin{proof}
		First assume $(2)$ and let $q \extends p$.
		By assumption choose $r \extends q$ such that
		$$
			\forall x \in [r] \ \chi(g_1^*(x), \dots, g_k^*(x), f_1, \dots, f_l).
		$$
		This is a $\Pi^1_1$-statement so that $\Pi^1_1$-absoluteness implies
		$$
			\SS^{\aleph_0} \forces \forall x \in [r] \ \chi(g_1^*(x), \dots, g_k^*(x), f_1, \dots, f_l).
		$$
		But we also have $r \forces ``s_\gen \in [r]"$ which implies
		$$
			r \forces \chi(g_1^*(s_\gen), \dots, g_k^*(s_\gen), f_1, \dots, f_l).
		$$
		Thus, we proved $(1)$.\\
		
		For the other direction we may assume that all integer quantifiers are in the front of $\chi$ and do an induction over the number of quantifiers of $\chi$.
		Let $q \extends p$.
		First, we have to consider the quantifier-free case.
		Then $\chi(v_1, \dots, v_k, w_1, \dots, w_l)$ only depends on finitely many values of $v_1, \dots, v_k, w_1, \dots, w_l$.
		So choose $N$ such that $\chi(v_1, \dots, v_k, w_1, \dots, w_l)$ only depends on the values of $v_1 \restr N, \dots, v_k \restr N, w_1 \restr N, \dots, w_l \restr N$.
		Since $g_1, \dots, g_k$ are codes by $\Pi^1_1$-absoluteness we have
		$$
			q \forces \exists K \ N \subseteq \dom(g_i(s_\gen \restr K \times K)) \text{ for all } i \in \simpleset{1, \dots, k},
		$$
		so choose $r \extends q$ and $K < \omega$ such that
		$$
			r \forces N \subseteq \dom(g_i(s_\gen \restr K \times K)) \text{ for all } i \in \simpleset{1, \dots, k}.
		$$
		Now, let $x \in [r]$ and define $r_x \extends r$ by
		$$
			r_x(n) := r(n)_{x(n) \restr K},
		$$
		which is well-defined since $x(n) \restr K \in r(n)$ follows from $x \in [r]$.
		But then
		$$
			r_x \forces s_\gen \restr K \times K = x \restr K \times K,
		$$
		so by choice of $r$ and $K$ we also have
		$$
			r_x \forces g_i^*(s_\gen) \restr N = g_i^*(x) \restr N \text{ for all } i \in \simpleset{1, \dots, k}.
		$$
		But $r_x \forces \chi(g_1^*(s_\gen), \dots, g_k^*(s_\gen), f_1, \dots, f_l)$, so by choice of $N$ we obtain
		$$
			r_x \forces \chi(g_1^*(x), \dots, g_k^*(x), f_1, \dots, f_l).
		$$
		Thus, we have proven $\forall x \in [r] \ \chi(g_1^*(x), \dots, g_k^*(x), f_1, \dots, f_l)$.
		
		Next, we have to prove the induction step.
		We handle the two different quantifier cases separately.
		First, assume that $\chi \equiv \exists n \psi$, so by assumption
		$$
			q \forces \exists n \ \psi(g_1^*(s_\gen), \dots, g_k^*(s_\gen), f_1, \dots, f_l, n).
		$$
		Choose $r \extends q$ and $n < \omega$ such that 
		$$
			r \forces \psi(g_1^*(s_\gen), \dots, g_k^*(s_\gen), f_1, \dots, f_l, n).
		$$
		By induction assumption choose $s \extends r$ such that
		$$
			\forall x \in [s] \ \psi(g_1^*(x), \dots, g_k^*(x), f_1, \dots, f_l, n).
		$$
		Then, we also have
		$$
			\forall x \in [s] \ \exists n \ \psi(g_1^*(x), \dots, g_k^*(x), f_1, \dots, f_l, n).
		$$
		Thus, we have proven $\forall x \in [s] \ \chi(g_1^*(x), \dots, g_k^*(x), f_1, \dots, f_l)$.
		
		Finally, assume that $\chi \equiv \forall n\psi$.
		We construct a fusion sequence $\seq{q_n}{n < \omega}$ below $q$ as follows.
		Set $q_0 := q$.
		Assume $q_n$ is defined.
		By induction assumption the set
		$$
			D_n := \set{r \extends q_n}{\forall x \in [r] \ \psi(g_1^*(x), \dots, g_k^*(x), f_1, \dots, f_l, n)}
		$$
		is dense open below $q_n$.
		By Lemma \ref{LEM_ExtendDenseOpen} take $q_{n + 1} \extends_n q_n$ such that $q_{n + 1} \restr \sigma \in D_n$ for all $\sigma$ suitable for $q_n$ and $n$.
		Notice, that
		$$
			[q_{n + 1}] = \bigcup_{\sigma \text{ suitable for } q_n \text{ and } n} [q_{n + 1} \restr \sigma],
		$$
		since $\set{q_{n + 1} \restr \sigma}{\sigma \text{ is suitable for } q_n \text{ and } n}$ is a maximal antichain below $q_{n + 1}$.
		But this implies
		$$
			\forall x \in [q_{n + 1}] \ \psi(g_1^*(x), \dots, g_k^*(x), f_1, \dots, f_l, n),
		$$
		for if $x \in [q_{n + 1}]$ choose $\sigma$ suitable for $q_n$ and $n$ such that $x \in [q_{n + 1} \restr \sigma]$.
		Then the desired conclusion follows from $q_{n + 1} \restr \sigma \in D_{n}$.
		Finally, let $r$ be the fusion of $\seq{q_n}{n < \omega}$.
		We claim that
		$$
			\forall x \in [r] \ \forall n \ \psi(g_1^*(x), \dots, g_k^*(x), f_1, \dots, f_l, n),
		$$
		so let $x \in [r]$ and $n < \omega$.
		Then $r \extends q_{n + 1}$, so that $x \in [r] \subseteq [q_{n + 1}]$.
		So by construction of $q_{n + 1}$
		$$
			\psi(g_1^*(x), \dots, g_k^*(x), f_1, \dots, f_l, n).
		$$
		Thus, we have proven $\forall x \in [r] \ \chi(g_1^*(x), \dots, g_k^*(x), f_1, \dots, f_l)$.
	\end{proof}

	Next, we will introduce the notion of an arithmetical type.
	In chapter \ref{SEC_Applications} we will verify that many different types of combinatorial families can be put into the following form:
	
	\begin{definition}\label{DEF_type}
		An arithmetical type $\t$ (of combinatorial family of reals) is a pair of sequences $\t = ((\psi_n)_{n < \omega}, (\chi_n)_{n < \omega})$ such that both $\psi_n(w_0, w_1, \dots, w_n)$ and $\chi_n(v,w_1, \dots, w_n)$ are arithmetical formulas in $n + 1$ real parameters.
		The domain of the type $\t$ is the set
		$$
			\dom(\t) := \set{\F \subseteq \P(\bairespace)}{\forall n < \omega \ \forall \simpleset{f_0, \dots, f_n} \in [\F]^{n + 1} \text{ we have } \psi_n(f_0, \dots, f_n)}
		$$
		If $\F \in \dom(\t)$ we say $\F$ is of type $\t$.
		Now, let $\F$ be of type $\t$.
		If a real $g$ satisfies
		$$
			\forall n < \omega \forall \simpleset{f_1, \dots, f_n} \in [\F]^n \ \chi_n(g, f_1, \dots, f_n),
		$$
		then we call $g$ an intruder for $\F$.
	\end{definition}
	
	Thus, the sequence $(\psi_n)_{n < \omega}$ defines what constitutes a family of that type and the sequence $(\chi_n)_{n < \omega}$ defines which reals constitute intruders.
	Note that in some specific examples these two properties coincide, e.g. for eventually different families both $\psi_1(w_0, w_1)$ and $\chi_1(v, w_1)$ assert the eventual difference of $w_0$ (or $v$, resp.) and $w_1$ (cf.\ Section~\ref{SECSUB_MED}).
	Also, if we want no restriction of what constitutes a family of type $\t$, we set $\psi_n :\equiv \top$ for all $n < \omega$ to obtain $\dom(\t) = \P(\P(\bairespace))$.
	This will be the case for the families considered in Section~\ref{SECSUB_OTHER}.
	
	\begin{lemma}\label{LEM_typeProperties}
		Let $t$ be an arithmetical type.
		Then we have the following:
		\begin{enumerate}
			\item $\emptyset$ is of type $\t$,
			\item If $\G$ is of type $\t$ and $\F \subseteq \G$, then $\F$ is of type $\t$,
			\item Let $\delta$ be a limit ordinal. If $\seq{\F_\alpha}{\alpha < \delta}$ is an increasing sequence of families of type $\t$, then also $\F := \bigcup_{\alpha < \delta} \F_\alpha$ is a family of type $\t$.
		\end{enumerate}
	\end{lemma}

	\begin{proof}
		(1) and (2) are obvious.
		For (3) let $\seq{\F_\alpha}{\alpha < \delta}$ is an increasing sequence of families of type and $n < \omega$ and $\simpleset{f_0, \dots, f_n} \in [\F]^{n + 1}$.
		Since $\delta$ is a limit we may choose $\alpha < \delta$ such that $\simpleset{f_0, \dots, f_n} \in [\F_\alpha]^{n + 1}$.
		But then $\psi_n(f_0, \dots, f_n)$ holds since $\F_\alpha$ is of type $\t$.
	\end{proof}

	Also note that since the $\psi_n$ are arithmetical formulas the notion of $\dom(\t)$ is absolute, i.e. for model of set theory $M \subseteq N$ we have that $\dom(\t)^M = \dom(\t)^N \cap M$.
	Analogously, the notion of an intruders is absolute.
	However, a family may have no intruders in $M$, but some in the larger model $N$.
	Thus, we define the following:
	
	\begin{definition}
		Given a forcing $\PP$ and a family $\F$ of type $\t$ we say that $\F$ is $\PP$-indestructible or $\PP$ preserves $\F$ iff $\PP$ forces that $\F$ has no intruders.
		In particular $\F$ has no intruders in the ground model.
		If $\F$ is indestructible by any countably supported product or iteration of Sacks-forcing of any length, we say that $\F$ is universally Sacks-indestructible.
	\end{definition}

	Now, equipped with these definitions we may now prove one of our main results:

	\begin{theorem}\label{THM_MainTheorem}
		Assume $\t$ is an arithmetical type and $\F$ is a $\SS^{\aleph_0}$-indestructible family of type $\t$.
		Then $\F$ is universally Sacks-indestructible.
	\end{theorem}

	\begin{proof}
		Let $\PP$ be the countably supported product or iteration of Sacks-forcing of any length and assume that $\F$ is not preserved by $\PP$.
		We may assume that the length of the product or iteration is at least $\aleph_0$.
		Choose $p \in \PP$ and a $\PP$-name $\dot{g}$ for a real such that
		$$
			p \forces_\PP \ \dot{g} \text{ is an intruder for } \F.
		$$
		By Lemma \ref{LEM_ContinuousReadingOfNames} choose $q \extends p$, a standard enumeration $\Sigma$ of $\dom(q)$ and a code $g$ such that
		$$
			q \forces_\PP \ \dot{g} = g^*(\dot{e}^{q, \Sigma}(s_\gen \restr \dom(q))).
		$$
		Since $\F$ is $\SS^{\aleph_0}$-indestructible we have
		$$
			\SS^{\aleph_0} \forces_{\SS^{\aleph_0}} \ g^*(s_\gen) \text{ is not an intruder for } \F,
		$$
		which is by assumption expressed by
		$$
			\SS^{\aleph_0} \forces_{\SS^{\aleph_0}} \ \exists n < \omega \ \exists \simpleset{f_1, \dots, f_n} \in [\F]^n \ \neg\chi_n(g^*(s_\gen), f_1, \dots, f_n).
		$$
		So choose $s \in \SS^{\aleph_0}$ and $\simpleset{f_1, \dots, f_n} \in [\F]^n$ such that
		$$
			s \forces_{\SS^{\aleph_0}} \ \neg\chi_n(g^*(s_\gen), f_1, \dots, f_n).
		$$
		Since $ \chi_n$ is an arithmetical formula by Lemma~\ref{LEM_ForcingToAbsolute} choose $t \extends s$ such that
		$$
			\forall x \in [t] \ \neg\chi_n(g^*(x), f_1, \dots, f_n).
		$$
		Now, this is a $\Pi^1_1$-formula, so we obtain
		$$
			\PP \forces_\PP \ \forall x \in [t] \ \neg\chi_n(g^*(x), f_1, \dots, f_n).
		$$
		Let $r$ be the preimage of $t$ under $\dot{e}^{q, \Sigma}$, i.e.\ we have $r \extends q$ and
		$$
			r \forces_\PP \ \dot{e}^{q, \Sigma}(s_\gen \restr \dom(q)) \in [t].
		$$
		Then we get
		$$
			r \forces_\PP \ \neg\chi_n(g^*(\dot{e}^{q, \Sigma}(s_\gen \restr \dom(q)), f_1, \dots, f_n),
		$$
		which yields
		$$
			r \forces_\PP \ \neg\chi_n(\dot{g}, f_1, \dots, f_n),
		$$
		contradicting that $r$ forces $\dot{g}$ to be an intruder for $\F$.
	\end{proof}
	
	\section{Constructing Sacks-indestructible maximal combinatorial families under $\sf{CH}$}\label{SEC_Applications}
	
	In this section we show that Theorem~\ref{THM_MainTheorem} may be applied to many different combinatorial families of reals considered in combinatorial set theory.
	We also provide a unified construction of a universally Sacks-indestructible witness under $\sf{CH}$ for many of these types of families if they satisfy the following following property:
	
	\begin{definition}\label{DEF_EliminatingIntrudersLemma}
		Let $\t$ be an arithmetical type.
		We say that $\t$ satisfies elimination of intruders and write $\eoi(\t)$ holds iff the following property is satisfied:
		If $\F$ is a countable family of type $\t$, $p \in \SS^{\aleph_0}$ and $\dot{g}$ is a name for a real such that
		$$
		p \forces \dot{g} \text{ is an intruder for } \F.
		$$
		Then there is $q \extends p$ and a real $f$ such that $\F \cup \simpleset{f}$ is of type $\t$ and
		\[
		q \forces \dot{g} \text{ is not an intruder for } \F \cup \simpleset{f}.
		\]
	\end{definition}

	As the name suggests $\eoi(\t)$ essentially asserts that for every countable family $\F$ of type $\t$ and every possible $\SS^{\aleph_0}$-name $\dot{g}$ for a possible intruder for $\F$ we may extend $\F$ by one element $f$ so that $\dot{g}$ is not an intruder for this extended family any more.
	Usually, in case that some $q \leq p$ forces $``\dot{g} = f \in V"$ the conclusion of $\eoi(\t)$ holds trivially for $q := p$ and $\F \cup \simpleset{f}$ or some other canonical extension of $\F$, so that we will additionally assume that $p \forces ``\dot{g} \notin V"$ when proving $\eoi(\t)$ for some arithmetical type $\t$.
	
	Finally, under $\sf{CH}$ we may now use continuous reading of names to enumerate all possible intruders for a given family in length $\aleph_1$ to iteratively construct a family of type $\t$ which is indestructible under $\SS^{\aleph_0}$.
	Note that in the construction of the following proof we make use of Lemma~\ref{LEM_typeProperties} multiple times.
	
	\begin{theorem}
		Assume $\sf{CH}$ and $\eoi(\t)$ holds.
		Then there is a $\SS^{\aleph_0}$-indestructible family of type $\t$.
	\end{theorem}

	\begin{proof}
		By $\sf{CH}$ we may enumerate all pairs $\seq{(p_\alpha, g_\alpha)}{\alpha <  \aleph_1}$ of elements $p \in \SS^{\aleph_0}$ and codes $g$.
		We construct an increasing and continuous sequence $\seq{\F_\alpha}{\alpha < \aleph_1}$ of families of type $\t$ as follows:
		
		Set $\F_0 := \emptyset$.
		Now assume $\F_\alpha$ is defined and $(p_\alpha, g_\alpha)$ is given.
		If we have
		$$
			p_\alpha \notforces g_\alpha^*(s_\gen) \text{ is an intruder for } \F_\alpha,
		$$
		then set $\F_{\alpha + 1} := \F_\alpha$.
		Otherwise, we have
		$$
			p_\alpha \forces g_\alpha^*(s_\gen) \text{ is an intruder for } \F_\alpha.
		$$
		Thus, by $\eoi(\t)$ choose $q \extends p_\alpha$ and a real $f$ such that the family $\F_{\alpha + 1} := \F_\alpha \cup \simpleset{f}$ is of type $\t$ and such that
		$$
			q \forces g_\alpha^*(s_\gen) \text{ is not an intruder for } \F_{\alpha + 1}.
		$$
		Finally, we set $\F := \bigcup_{\alpha < \aleph_1} \F_\alpha$.
		We show that $\F$ is $\SS^{\aleph_0}$-indestructible.
		
		Assume not.
		Choose $p \in \SS^{\aleph_0}$ and a $\SS^{\aleph_0}$-name $\dot{g}$ for a real such that
		$$
			p \forces \dot{g} \text{ is an intruder for } \F.
		$$
		By Lemma \ref{LEM_ContinuousReadingOfNames} choose $q \extends p$ and a code $g:\finseqcantorspace \to \finbairespace$ so that
		$$
			q \forces \dot{g} = g^*(s_\gen).
		$$
		But then, the pair $(q, g)$ appeared at some step $\alpha$ in our enumeration and we extended $\F_\alpha$ using $\eoi(\t)$ at that step.
		Thus, there is $r \extends q$ with
		$$
			r \forces g^*(s_\gen) \text{ is not an intruder for } \F_{\alpha + 1},
		$$
		which yields the contradiction
		\[
			r \forces \dot{g} \text{ is not an intruder for } \F.
		\]
	\end{proof}
	
	Remember that by Theorem \ref{THM_MainTheorem} we have that $\SS^{\aleph_0}$-indestructibility actually implies universal Sacks-indestructibility, so we proved the following theorem:
	
	\begin{theorem}\label{THM_SacksUniversal}
		Assume $\sf{CH}$ and $\eoi(\t)$ holds.
		Then there is a universally Sacks-indestructible family of type $\t$.
	\end{theorem}
	
	\begin{proof}
		This follows directly by composing the previous theorem with Theorem~\ref{THM_MainTheorem}.
	\end{proof}

	Thus, the remaining objective for this paper will be to verify that many different types of combinatorial families fall into our framework of arithmetical types and to prove elimination of intruders for various different types of families.
	
	\subsection{Mad families}\label{SECSUB_MAD}
	
	We start with one of the most common example for combinatorial sets of reals - mad families.
	For this case we will explain a bit more explicitly how to phrase the definitions in a suitable way so that we may apply Theorem~\ref{THM_MainTheorem}.
	For the subsequent types of families the discussion we will just mention the necessary coding arguments, but we omit the analogous details.
	Remember the following definition:
	
	\begin{definition}
		A family $\A$ of infinite subsets of $\omega$ is almost disjoint (a.d.) iff $A \cap B$ is finite for all $A \neq B \in \A$ and for all $\A_0 \in \finsubset{\A}$ we have that $\omega \setminus \bigcup \A_0$ is infinite.
		$\A$ is called maximal (mad) iff it is maximal with respect to inclusion.
		The corresponding cardinal characteristic is the almost disjointness number $\a$:
		$$
			\a := \min{\set{\left|\A\right|}{\A \text{ is a mad family}}}.
		$$
	\end{definition}
	
	The second property in the definition of an a.d.\ family is vacuous if $\A$ is infinite, but in the finite case it allows us to exclude maximal finite a.d.\ families.	
	Note that Theorem~\ref{THM_MainTheorem} is applied to combinatorial families on the Polish space $\bairespace$, so to be more precise we code a.d.\ families in that Polish space:
	
	\begin{definition}
		A family $\F$ of reals codes an a.d.\ family iff every $f \in \F$ codes an infinite subset of $\omega$, i.e.\ $f$ is a strictly increasing function, $\ran(f) \cap \ran(g)$ is finite for all $f \neq g \in \F$ and for all $\F_0 \in \finsubset{\F}$ we have that $\omega \setminus \bigcup_{f \in \F_0} \ran(f)$ is infinite.
		$\F$ is called maximal iff it is infinite and maximal {w.r.t.} inclusion.
	\end{definition}
	
	\begin{proposition}
		Coded mad families are an arithmetical type.
	\end{proposition}

	\begin{proof}
		We define the formula $\psi_0(w_0)$ to be
		$$
			\forall n \forall m (n < m \text{ implies } w_0(n) < w_0(m)),
		$$
		expressing \textquoteleft $w_0$ codes an infinite subset of $\omega$'.
		Further, $\psi_1(w_0, w_1)$ is defined as
		$$
			\exists N \forall n \forall m (n > N \text{ implies } w_0(n) \neq w_1(m)),
		$$
		expressing \textquoteleft $\ran(w_0) \cap \ran(w_1)$ is finite'.
		Finally, for $n > 1$ we define $\psi_n(w_0, \dots, w_n)$ by
		$$
			\forall N \exists n \forall m (n > N \text{ and } \bigwedge_{i = 0}^n w_i(m) \neq n),
		$$
		expressing \textquoteleft $\omega \setminus \bigcup_{i = 0}^n \ran(w_i)$ is infinite'.
		Analogously, we define the formula $\chi_0(v)$ to be
		$$
			\forall n \forall m (n < m \text{ implies } v(n) < v(m)),
		$$
		expressing \textquoteleft $v$ codes an infinite subset of $\omega$'.
		Finally, we define $\chi_1(v,w_1)$ by
		$$
			\exists N \forall n \forall m (n > N \text{ implies } v(n) \neq w_1(m)),
		$$
		expressing \textquoteleft$\ran(v) \cap \ran(w_1)$ is finite'
		and set $\chi_n :\equiv \top$ for all $n > 1$.
		Clearly, with respect to Definition~\ref{DEF_type} this exactly captures our definition of a coded mad family.
		Thus, coded mad families are an arithmetical type.
	\end{proof}
	
	Thus, explicitly for mad families Theorem~\ref{THM_MainTheorem} implies that $\SS^{\aleph_0}$-indestructibility for families which code a mad family implies universal Sacks-indestructibility.
	Since the coding is absolute this is equivalent to the respective version without coding:
	
	\begin{corollary}\label{COR_MAD}
		Every $\SS^{\aleph_0}$-indestructible mad family is universally Sacks-indestructible.
	\end{corollary}
	
	Next, we prove $\eoi(\a)$, i.e.\ elimination of intruders for mad families.
	To be more precise we would have to prove elimination of intruders for families which code mad families, but since our coding is absolute these are easily seen to be equivalent.
	
	\begin{lemma}
		Let $\A$ be a countable a.d.\ family, $p \in \SS^{\aleph_0}$ and $\dot{B}$ be a name for an infinite subset of $\omega$ such that
		$$
			p \forces \dot{B} \notin V \text{ and } \A \cup \simpleset{\dot{B}} \text{ is an a.d.\ family}.
		$$
		Then there is $q \extends p$ and an infinite subset $A$ of $\omega$ such that $\A \cup \simpleset{A}$ is an a.d.\ family and
		$$
			q \forces \A \cup \simpleset{A, \dot{B}} \text{ is not an a.d.\ family}.
		$$
	\end{lemma}
	
	\begin{proof}
		If $\A$ is finite we have that $\omega \setminus \bigcup \A$ is infinite, so let $D \cup E$ be a partition of  $\omega \setminus \bigcup A$ into two infinite sets.
		By assumption for every $A \in \A$ we have
		$$
			p \forces \exists k < \omega \ A \cap \dot{B} \subseteq k.
		$$
		Since $\A$ is finite choose $q \extends p$ and $k < \omega$ such that for all $A \in \A$
		$$
			q \forces A \cap \dot{B} \subseteq k.
		$$
		This implies that
		$$
			q \forces \dot{B} \setminus k \subseteq \omega \setminus \bigcup \A = D \cup E.
		$$
		Since $\dot{B}$ is a name for an infinite subset of $\omega$ there is $r \extends q$ such that
		$$
			r \forces D \cap \dot{B} \text{ is infinite} \quad \text{or} \quad r \forces E \cap \dot{B} \text{ is infinite}.
		$$
		W.l.o.g. assume the first case holds.
		Then $\A \cup \simpleset{D}$ is an a.d.\ family and
		$$
			r \forces \A \cup \simpleset{D, \dot{B}} \text{ is not an a.d.\ family}.
		$$
		Now, assume that $\A$ is infinite, so enumerate $\A = \set{A_n}{n < \omega}$.
		We construct a fusion sequence $\seq{p_n}{n < \omega}$ below $p$ and a sequence $\seq{k_n < \omega}{n < \omega}$ as follows.
		Set $p_0 := p$.
		Now, assume $p_n$ is defined.
		By assumption the set
		$$
			D_n := \set{q \extends p_n}{\exists k < \omega \ q \forces ``A_n \cap \dot{B} \subseteq k"}
		$$
		is dense open below $p_n$.
		By Lemma \ref{LEM_ExtendDenseOpen} take $p_{n + 1} \extends_n p_n$ such that $p_{n + 1} \restr \sigma \in D_n$ for all $\sigma$ suitable for $p_n$ and $n$, witnessed by $k_\sigma < \omega$.
		Set $k_n := \max \set{k_\sigma}{\sigma \text{ suitable for } p_n \text{ and } n}$, so that
		$$
			p_{n + 1} \forces A_n \cap \dot{B} \subseteq k_n.
		$$
		Let $q_0$ be the fusion of $\seq{p_n}{n < \omega}$.
		We define a second fusion sequence $\seq{q_n}{n < \omega}$ below $q_0$ and a sequence $\seq{a_n \in \finsubset{\omega}}{n < \omega}$ of disjoint sets.
		Assume $q_n$ is defined and choose $K < \omega$ with $K > a_m, k_m$ for all $m < n$.
		Since $q_n \forces `` \dot{B} \text{ is infinite}"$ the set 
		$$
			E_n := \set{q \extends q_n}{\exists k > K \ q \forces `` k \in \dot{B} "}
		$$
		is dense open below $p_n$.
		Again by Lemma \ref{LEM_ExtendDenseOpen} take $q_{n + 1} \extends_n q_n$ such that $q_{n + 1} \restr \sigma \in E_n$ for all $\sigma$ suitable for $q_n$ and $n$, witnessed by $k_\sigma$.
		Set $a_n := \set{k_\sigma}{\sigma \text{ suitable for } q_n \text{ and } n}$, so that
		$$
			q_{n + 1} \forces \dot{B} \cap a_m \neq \emptyset.
		$$
		By choice of $k_m$ for $m < n$ we also have that $a_n \cap A_m = \emptyset$ for all $m < n$.
		Finally, let $q$ be the fusion of $\seq{q_n}{n < \omega}$ and set $A := \bigcup_{n < \omega} a_n$.
		Then we have that $\A \cup \simpleset{A}$ is an a.d.\ family and
		$$
			q \forces A \cap \dot{B} \text{ is infinite}.
		$$
		Since $q \forces `` \dot{B} \notin V"$ we have $q \forces ``A \neq \dot{B}"$, so we obtain
		\[
			q \forces \A \cup \simpleset{A, \dot{B}} \text{ is not an a.d.\ family}.
		\]
	\end{proof}
	
	Thus, we proved $\eoi(\a)$ and we get the following new result as an instance of Theorem~\ref{THM_SacksUniversal} for mad families:
	
	\begin{corollary}
		Assume $\sf{CH}$.
		Then there is a universally Sacks-indestructible mad family.
	\end{corollary}
	
	\subsection{Eventually different families}\label{SECSUB_MED}
	Next, we will briefly consider maximal eventually different families.
	This case is especially noteworthy as the idea for the Main Theorem~\ref{THM_MainTheorem} originates from Fischer's and Schrittesser's construction of a universally Sacks-indestructible eventually different family in \cite{FischerSchrittesser}.
	The main definition is the following:
	
	\begin{definition}
		A family of reals $\F \subseteq \bairespace$ is eventually different (e.d.) iff for all $f \neq g \in \F$ there is $N < \omega$ such that $f(n) \neq g(n)$ for all $n > N$.
		$\F$ is called maximal (med) iff it is maximal w.r.t. to inclusion.
		The corresponding cardinal characteristic is $\aE$:
		\[
			\aE := \min{\set{\left|\F\right|}{\F \text{ is a med family}}}.
		\]
	\end{definition}

	\begin{proposition}
		Med families are an arithmetical type.
	\end{proposition}

	\begin{proof}
		We set $\psi_n :\equiv \top$ for $n \neq 1$ and define the formula $\psi_1(w_0, w_1)$ to be
		$$
			\exists N \forall n (n > N \text{ implies } w_0(n) \neq w_1(n)),
		$$
		expressing \textquoteleft $w_0$ and $w_1$ are eventually different'.
		Analogously, we set $\chi_n \equiv \top$ for $n \neq 1$ and define the formula $\chi_1(v, w_1)$ to express \textquoteleft $v$ and $w_1$ are eventually different'.
		Thus, med families is an arithmetical type.
	\end{proof}

	Hence, Theorem~\ref{THM_MainTheorem} implies for this instance:
	
	\begin{corollary}
		Every $\SS^{\aleph_0}$-indestructible med family is universally Sacks-indestructible.
	\end{corollary}

	Note that Fischer and Schrittesser essentially proved $\eoi(\aE)$ in \cite{FischerSchrittesser}.
	More precisely, combining $\eoi(\aE)$ with Lemma~\ref{LEM_ForcingToAbsolute} exactly yields their Lemma~7, which they used to construct their universally Sacks-indestructible med family.
	For completeness we present a direct proof of $\eoi(\aE)$, which will also serve as a rough template for the corresponding lemma for maximal cofinitary groups in one of the following sections.
	
	\begin{lemma}\label{LEM_EILforED}
		Let $\F$ be a countable e.d.\ family, $p \in \SS^{\aleph_0}$ and $\dot{g}$ be a name for a real such that
		$$
			p \forces \dot{g} \notin V \text{ and } \F \cup \simpleset{g} \text{ is an e.d.\ family}.
		$$
		Then there is $q \extends p$ and a real $\text{f}$ such that $\F \cup \simpleset{f}$ is an e.d.\ family and
		$$
			q \forces \F \cup \simpleset{f, \dot{g}} \text{ is not an e.d.\ family}.
		$$
	\end{lemma}
	
	\begin{proof}
		Enumerate $\F = \set{f_n}{n < \omega}$.
		We construct a fusion sequence $\seq{p_n}{n < \omega}$ below $p$ and a sequence $\seq{k_n < \omega}{n < \omega}$ as follows.
		Set $p_0 := p$.
		Now, assume $p_n$ is defined.
		By assumption the set
		$$
			D_n := \set{q \extends p_n}{\exists k < \omega \ q \forces ``\forall l > k \ f_n(l) \neq \dot{g}(l)"}
		$$
		is dense open below $p_n$.
		By Lemma \ref{LEM_ExtendDenseOpen} take $p_{n + 1} \extends_n p_n$ such that $p_{n + 1} \restr \sigma \in D_n$ for all $\sigma$ suitable for $p_n$ and $n$, witnessed by $k_\sigma < \omega$.
		Set $k_n := \max \set{k_\sigma}{\sigma \text{ suitable for } p_n \text{ and } n}$, so that
		$$
			p_{n + 1} \forces \forall l > k_n \ f_n(l) \neq \dot{g}(l).
		$$
		Let $q_0$ be the fusion of $\seq{p_n}{n < \omega}$.
		We define a second fusion sequence $\seq{q_n}{n < \omega}$ below $q_0$ and a sequence of partial functions $\seq{h_n}{n < \omega}$ with the following properties:
		\begin{enumerate}
			\item The sequence $\seq{\dom(h_n)}{n < \omega}$ is an increasing interval partition of $\omega$,
			\item We have $k_n \leq \max (\dom(h_n))$,
			\item For all $m < n$ and $l \in \dom(h_n)$ we have $f_m(l) \neq h_n(l)$,
			\item $q_{n + 1} \forces \exists l \in \dom(h_n) \ h_n(l) = \dot{g}(l)$.
		\end{enumerate}
		This finishes the proof, for if $q$ is the fusion of $\seq{q_n}{n < \omega}$, then $f := \bigcup_{n < \omega} h_n$ is total function by $(1)$.
		Furthermore, $(3)$ implies that $\F \cup \simpleset{f}$ is an e.d.\ family and $(4)$ implies that
		$$
			q \forces f \text{ and } \dot{g} \text{ are not eventually different}.
		$$
		Since $q \forces ``\dot{g} \notin V"$ we have $q \forces ``f \neq \dot{g}$, so we obtain
		$$
			q \forces \F \cup \simpleset{f, \dot{g}} \text{ is not an e.d.\ family}.
		$$
		For the fusion construction assume $q_n$ is defined and enumerate with $\seq{\sigma_i}{i < N}$ the set of all suitable functions $\sigma$ for $q_n$ and $n$.
		Let $I$ be the interval above $\bigcup_{m < n} \dom(h_m)$ of size $\max(k_n, N)$.
		We construct $h_n$ with $\dom(h_n) = I$, so that $(1)$ to $(4)$ hold.
		The choice of $I$ already implies that $(1)$ and $(2)$ are satisfied.
		Note, that the set
		$$
			E_n := \set{q \extends p_n}{q \text{ decides } \dot{g} \text{ on } I "}
		$$
		is dense open below $p_n$.
		Again by Lemma \ref{LEM_ExtendDenseOpen} take $q_{n + 1} \extends_n q_n$ such that $q_{n + 1} \restr \sigma \in E_n$ for all $\sigma$ suitable for $q_n$ and $n$, witnessed by the decision $g_\sigma:I \to \omega$.
		We define $h_n$ for $i < N$ by
		$$
			h_n(\min(I) + i) := g_{\sigma_i}(\min(I) + i).
		$$
		For all other $l \in I$ with $l \geq \min(I) + N$ we define $h_n(l)$ arbitrarily so that $(3)$ is satisfied.
		By construction of $q_{n + 1}$ and $h_n$ we have for all $i < N$
		$$
			q_{n + 1} \restr \sigma_i \forces h_n(\min(I) + i) = g_{\sigma_i}(\min(I) + i) = \dot{g}(\min(I) + i),
		$$
		i.e.\ $(4)$ is satisfied.
		For $(3)$ let $m < n$ and $l \in I$.
		If $l \geq \min(I) + N$ there is nothing to show, so let $i < N$.
		By induction assumption for $m < n$ we have $k_m < \min(I) + i$, so that by choice of $k_m$
		$$
			q_{n + 1} \forces f_m(\min(I) + i) \neq \dot{g}(\min(I) + i).
		$$
		On the other hand we have
		$$
			q_{n + 1} \restr \sigma_i \forces \dot{g}(\min(I) + i) = h_n(\min(I) + i),
		$$
		which implies $f_m(\min(I) + i) \neq h_n(\min(I) + i)$.
	\end{proof}

	Thus, we proved $\eoi(\aE)$ and as before, Theorem~\ref{THM_SacksUniversal} yields the following theorem, which corresponds to Theorem~9 in \cite{FischerSchrittesser}:
	
	\begin{corollary}
		Assume $\sf{CH}$.
		Then there is a universally Sacks-indestructible med family.
	\end{corollary}
	
	\subsection{Partitions of Baire space into compact sets}\label{SECSUB_ADFS}
	
	In this section we want to consider partitions of Baire space into compact sets.
	Recently, the authors constructed a universally Sacks-indestructible such partition in \cite{FischerSchembecker} with the same techniques, so that we only summarize how our results generalize that construction.
	The key observation is that partitions of Baire space into compact sets are in one-to-one correspondence with the following types of families:
	
	\begin{definition}
		A family $\mathcal{T}$ of finitely splitting trees on $\omega$ is called an almost disjoint family of finitely splitting trees (or an a.d.f.s.\ family) iff $S$ and $T$ are almost disjoint, i.e.\ $S \cap T$ is finite for all $S \neq T \in \T$.
		It is called maximal iff for all $f \in \bairespace$ there is $T \in \T$ with $f \in [T]$.
		Here, $[T]$ denotes the set of branches trough $T$.
		The corresponding cardinal characteristic is $\aT$:
		\[
			\aT := \min{\set{\left|\T\right|}{\F \text{ is a maximal a.d.f.s.\ family}}}.
		\]
	\end{definition}
	
	\begin{remark}
		There is a one-to-one correspondence between non-empty compact subsets of $\bairespace$ and finitely splitting trees on $\omega$ given by the following maps:
		
		Given a finitely splitting tree $T$ on $\omega$ its set of branches $[T]$ is a non-empty compact subset~of~$\bairespace$.
		Conversely, given a non-empty compact subset $C$ of $\bairespace$ we define a finitely splitting tree by
		$$
			T_C := \set{s \in \bairespace}{\exists f \in C \ s \subseteq f}.
		$$
		It is easy to check that these maps are inverse to each other.
	\end{remark}
	
	Notice that by König's lemma for finitely splitting trees $S$ and $T$ we have that $S$ and $T$ are almost disjoint iff $[T] \cap [S] = \emptyset$.
	Thus, using the above identification of finitely splitting trees and non-empty compact subsets of $\bairespace$, we can also identify maximal a.d.f.s.\ families with partitions of $\bairespace$ into non-empty compact sets.
	
	\begin{proposition}
		Maximal a.d.f.s.\ families are an arithmetical type.
	\end{proposition}

	\begin{proof}
		As for mad families, finitely splitting trees do not live in $\bairespace$, so we have to use an arithmetically definable coding of sequences of natural numbers (for example using prime decomposition).
		This means that there is an injection $\textquoteleft\code\text{'}:\finbairespace \to \omega$ such that the statements \textquoteleft$n$ is the code for some $s_n \in \finbairespace$', \textquoteleft $s_n \subsequence s_m$' and \textquoteleft $\code(v \restr n) = m$` are definable by arithmetical formulas $\varphi_0(n), \varphi_1(n,m)$ and $\varphi_2(v, n, m)$.
		Then, we define $\psi_0(w_0)$ by
		\begin{align*}
			&\forall n (w_0(n) = 1 \text{ implies } \varphi_0(n))\\
			\text{ and } &\forall n,m ((w_0(n) = 1 \text{ and } s_m \subsequence s_n) \text{ implies } w_0(m) = 1)\\
			\text{ and } &\forall n (w_0(n) = 1 \text{ implies } \exists m (m \neq n, s_n \subsequence s_m \text{ and } w_0(m) = 1))\\
			\text{ and } &\forall n (w_0(n) = 1 \text{ implies } \exists M \forall m ((m > M, s_n \subsequence s_m \text{ and } w_0(m) = 1)\\
			&\qquad \text{ implies } \exists k (k \neq n, k \neq m \text{ and } s_n \subsequence s_k \subsequence s_m))),
		\end{align*}
		expressing \textquoteleft $\code^{-1}[w_0^{-1}[\simpleset{1}]]$ is a finitely splitting tree'.
		Further, we define $\psi_1(w_0, w_1)$ by
		$$
			\exists N \forall N (n > N \text{ implies } (w_0(n) \neq 1 \text{ or } w_1(n) \neq 1)),
		$$
		expressing \textquoteleft$\code^{-1}[w_0^{-1}[\simpleset{1}]]$ and $\code^{-1}[w_1^{-1}[\simpleset{1}]]$ are almost disjoint'.
		Set $\psi_n :\equiv \top$ for all $n > 1$.
		Analogously, set $\chi_n :\equiv \top$ for all $n \neq 1$.
		Finally, we set $\chi_1(v, w_1)$ as
		$$
			\exists n \forall m (\varphi_2(v, n, m) \text{ implies } w_1(m) = 0),
		$$
		expressing \textquoteleft$v \notin [\code^{-1}[w_0^{-1}[\simpleset{1}]]]$'.
		Thus, maximal a.d.f.s.\ families are an arithmetical type.
	\end{proof}
	
	As usual, Theorem~\ref{THM_MainTheorem} yields the following new result:
	
	\begin{corollary}
		Every $\SS^{\aleph_0}$-indestructible a.d.f.s.\ family (partition of Baire space into compact sets) is universally Sacks-indestructible.
	\end{corollary}

	In order to construct a universally Sacks-indestructible partition of Baire space into compact sets, the authors proved $\eoi(\aG)$ in Lemma~4.8~in~\cite{FischerSchembecker}, that is:
	
	\begin{lemma}
		Let $\T$ be a countable a.d.f.s.\ family, $p \in \SS^{\aleph_0}$ and $\dot{g}$ be a name for a real such that
		$$
			p \forces \dot{g} \notin V \text{ and } \ \forall T \in \T \ \dot{g} \notin [T].
		$$
		Then there is $q \extends p$ and a finitely splitting tree $S$ such that $\T \cup \simpleset{S}$ is an a.d.f.s.\ family and
		\[
			q \forces \dot{g} \in [S].
		\]
	\end{lemma}

	Thus, together with Theorem \ref{THM_SacksUniversal} we also obtain Theorem~4.17 in \cite{FischerSchembecker}:
	
	\begin{corollary}
		Assume $\sf{CH}$.
		Then there is a universally Sacks-indestructible a.d.f.s.\ family (partition of Baire space into compact sets).
	\end{corollary}
	
	Finally, by Lemma~4.20 in \cite{FischerSchembecker} exactly the same arguments also work for almost disjoint families of nowhere dense trees, which correspond to partitions of $\cantorspace$ into closed sets.
	
	\subsection{Maximal cofinitary groups}\label{SECSUB_MCG}
	
	Next on our list is a more algebraic example, namely maximal cofinitary groups.
	As with the other type of families we first show that also maximal cofinitary groups are an arithmetical type.
	Very similarly to med families we will also construct a universally Sacks-indestructible maximal cofinitary group.
	We will also carefully set up nice words again, since there are some inaccuracies in the literature.
	For the remainder of this section fix a set $A$, which will serve as an index set.
	
	\subsubsection{Definitions and notations}
	
	We denote with $W_A$ the set of all reduced words in the language $A^{\pm1} := \set{a^{i}}{a \in A \text{ and } i = \pm1}$.
	$W_A$ is a group with concatenate-and-reduce as group operation.
	$W_A$ satisfies the universal property of the free group generated by $A$, i.e.\ for every group $G$ any map $\rho:A \to G$ uniquely extends to a group homomorphism $\hat{\rho}: W_A \to G$.
	
	Analogously, with $M_A$ we denote the set of all words in the language $A^{\pm1}$.
	$M_A$ is a monoid with concatenate as monoid operation.
	$M_A$ satisfies the universal property of the free monoid generated by $A^{\pm1}$, i.e.\ for every monoid $M$ any map $\rho:A^{\pm1} \to M$ uniquely extends to a monoid homomorphism $\hat{\rho}:M_A \to M$.
	
	$\Sinf$ denotes the set of all permutations of $\omega$ and $\Sinffin$ the set of all finite partial injections $f:\omega \smash{\partialto} \omega$.
	For $f \in \Sinffin$ and $n < \omega$ we write $f(n) \defined$ iff $n \in \dom(f)$ and $f(n) \undefined$ otherwise.
	Set $\Sinfplus := \Sinf \cup \Sinffin$.
	Then $\Sinf$ is a group with concatenation, whereas $\Sinfplus$ is only a monoid.
	In fact, $\Sinf$ are exactly the invertible elements of $\Sinfplus$.
	For $f \in \Sinfplus$ let
	$$
		\fix(f) := \set{n < \omega}{f(n) = n}
	$$
	be the set of fixpoints of $f$.
	Further, define the set of all cofinitary permutations
	$$
		\cofin(\Sinf) := \set{f \in \Sinf}{\fix(f) \text{ is finite}}.
	$$
	We say $\rho:A \to \Sinf$ induces a cofinitary representation iff the induced map $\hat{\rho}:W_A \to \Sinf$ satisfies $\ran(\hat{\rho}) \subseteq \cofin(\Sinf) \cup \simpleset{\id}$.
	Analogously, we call a subgroup $G$ of $\Sinf$ cofinitary iff $G \subseteq \cofin(\Sinf) \cup \simpleset{\id}$.
	Note that $G$ is cofinitary iff there is cofinitary representation $\hat{\rho}:W_A \to \Sinf$ with $\ran(\hat{\rho}) = G$.
	$G$ is called maximal iff it is maximal {w.r.t.} to inclusion.
	Analogously, a cofinitary representation $\rho$ is called maximal iff $\ran(\hat{\rho})$ is a maximal cofinitary group.
	
	If $\rho: A \to \Sinfplus$ its induced map $\hat{\rho}:M_A \to \Sinfplus$ is defined as follows.
	Define $\rho^{\pm1}:A^{\pm1} \to \Sinfplus$ by
	$$
	\rho^{\pm1}(a^i) := \begin{cases}
		\rho(a) & \text{if i = 1}\\
		\rho(a)^{-1} & \text{if i = -1}
	\end{cases}
	$$
	Then let $\hat{\rho}:M_A \to \Sinfplus$ be the induced map for $\rho^{\pm1}$  given by the universal property of $M_A$.
	Note that this map coincides with the induced map given by the universal property of $W_A$ in case that $\rho:A \to \Sinf$.

	From now on we usually identify $\rho$ and its induced map $\hat{\rho}$, write $\epsilon$ for the empty word and $|w|$ for the length of a word $w$.
	Further, we fix $x \notin A$ and if $\rho:A \to \Sinf$ induces a cofinitary representation and $f \in \Sinfplus$ we write $\rho[f]$ for $\rho \cup (x,f)$.
	
	\subsubsection{Arithmetical definability}
	
	\begin{proposition}
		Maximal cofinitary groups are an arithmetical type.
	\end{proposition}

	\begin{proof}
		Define the formula $\psi_0(w_0)$ to be
		$$
			(\forall n \forall m \ v(n) = v(m) \text{ implies } n = m) \text{ and } (\forall n \exists m \ v(m) = n),
		$$
		expressing \textquoteleft $w_0 \in \Sinf$'.
		Next, we fix an enumeration $\seq{u_n}{1 < n < \omega}$ of $W_{\NN}$, so that $u_n$ only contains natural numbers up to $n$ as letters.
		For $n > 0$ we define $\psi_n(w_0, \dots, w_n)$ to be
		\begin{align*}
			\forall k_0 \exists k_1, \dots, \exists k_{\left|u_n\right|} &(\bigwedge_{i = 0}^{\left|u_n\right| - 1}\pi_i(v, w_1, \dots, w_n, k_i, k_{i + 1}) \text{ and } k_0 = k_{\left|u_n\right|})\\
			\text{or } \exists K \forall k_0,  \exists k_1, \dots, \exists k_{\left|u_n\right|} \ &(K < k_0 \text{ implies } (\bigwedge_{i = 0}^{\left|u_n\right| - 1}\pi_i(v, w_1, \dots, w_n, k_i, k_{i + 1}) \text{ and }k_0 \neq k_{\left|u_n\right|})),
		\end{align*}
		expressing \textquoteleft$\rho(u_n) = \id$ or $\rho(u_n)$ has finitely many fixpoints', where $\rho$ is defined by $m \mapsto w_m$.
		Here, for $u_n := y_{\left|u_n\right| - 1} \dots y_0$ the formula $\pi_i(v, w_1, \dots, w_n, k_i, k_{i + 1})$ is defined as
		$$
		\begin{cases}
			k_{i + 1} = w_{m}(k_i) & \text{if } y_i = m \text{ for } m \in \NN,\\
			k_{i} = w_{m}(k_{i + 1}) & \text{if } y_i = m^{-1} \text{ for } m \in \NN,
		\end{cases}
		$$
		expressing \textquoteleft$\rho(y_i)(k_i) = k_{i + 1}$'.
		Analogously, we define $\chi_0(v)$ as
		$$
			(\exists n \exists m \ n \neq m \text{ and } v(n) = v(m)) \text{ or } (\exists n \forall m \ v(m) \neq n),
		$$
		expressing $v \notin \Sinf$.
		Finally, fix an enumeration $\seq{u_n}{1 < n < \omega}$ of $W_{\NN}$, so that $u_n$ only contains natural numbers up to $n$ as letters.
		Analogously, for $n > 0$ there is an arithmetical formula $\chi_n(v, w_1, \dots, w_n)$ expressing
		$$
			\rho(u_n) = \id \text{ or } \rho(u_n) \text{ has finitely many fixpoints},
		$$
		where $\rho$ is defined by $0 \mapsto v$ and $m \mapsto w_m$ for $m > 0$.
		Thus, maximal cofinitary groups are an arithmetical type.
	\end{proof}
	
	Thus, using Theorem~\ref{THM_MainTheorem} we obtain the following new fact:
	
	\begin{corollary}
		Every $\SS^{\aleph_0}$-indestructible m.c.g.\ is universally Sacks-indestructible.
	\end{corollary}
	
	\subsubsection{Nice words and range/domain extension}
	
	In this subsection we reintroduce nice words and reprove their corresponding range and domain extension lemmata, which are the crucial tools to approximate elements of $\Sinf$ by finite segments.
	First, we prove that if we are only interested in the number of fixpoints of $\rho(w)$, then we can also equivalently consider any cyclic permutation of $w$.
	
	\begin{proposition}\label{prop_same_number_of_fixpoints}
		Let $\rho:A \to \Sinf$ and $u,v \in W_A$.
		Then $|\fix(\rho(uv))| = |\fix(\rho(vu))|$.
		In fact, there is a bijection given by $\pi:n \mapsto \rho(v)(n)$.
	\end{proposition}
	
	\begin{proof}
		$\pi$ is injective as $\rho(v) \in \Sinf$.
		Let $n \in \fix(\rho(uv))$.
		Then
		$$
			\rho(vu)(\pi(n)) = \rho(vu)(\rho(v)(n)) = \rho(vuv)(n) = \rho(v)(\rho(uv)(n)) = \rho(v)(n) = \pi(n),
		$$
		i.e.\ $\pi(n) \in \fix(vu)$.
		Now, let $n \in \fix(\rho(vu))$, then $\rho(v^{-1})(n) \in \fix(\rho(uv))$ since
		$$
			\rho(uv)(\rho(v^{-1})(n)) = \rho(v^{-1}vuvv^{-1})(n) = \rho(v^{-1})(\rho(vu)(n)) = \rho(v^{-1})(n).
		$$
		But $\pi(\rho(v^{-1}(n))) = \rho(v)(\rho(v^{-1})(n)) = \rho(vv^{-1})(n) = n$, thus $\pi$ is surjective.
	\end{proof}
	
	From now on assume that $\rho:A \to \Sinf$ induces a cofinitary permutation.
	
	\begin{definition}
		We call two words $w,v \in W_{A \cup \simpleset{x}}$ equivalent (with respect to $\rho$) and write $w \sim_\rho v$ iff $[x]_{\sim_\rho} = [v]_{\sim_\rho}$, where $[x]_{\sim_\rho}$ is the equivalence class of $w$ in $W_A / \ker(\rho)$.
	\end{definition}
	
	\begin{definition}
		Define $W^0_{\rho,x} := W_A \setminus \ker(\rho)$.
		For $n > 0$ define $W^n_{\rho, x}$ to be the set of all reduced words $w \in W_{A \cup \simpleset{x}}$ of the form $w = x^{\pm n}$ or
		$$
			w = u_l x^{k_l} u_{l-1} x^{k_{l-1}} \dots u_1 x^{k_1} u_0x^{k_0}
		$$
		
		for some $l < \omega$ and $u_i \in W^0_{\rho,x}$, $k_i \in \ZZ \setminus \simpleset{0}$ for $i \leq l$ and $\sum_{i = 0}^{l}|k_i| = n$.
		Finally, we set $W_{\rho, x} := \bigcup_{n > 0} W^n_{\rho,x}$.
		We call $W_{\rho, x}$ the set of all nice words (with respect to $\rho$).
		Further, we say a reduced word $w$ is split into $uv$ iff $w = uv$ without reducing.
	\end{definition}
	
	\begin{lemma} \label{lem_split_into_nice_word}
		Every word $w \in W_{A \cup \simpleset{x}}$ can be split as $w = uv$ for $u,v \in W_{A \cup \simpleset{x}}$ such that $vu$ is equivalent to a word in $W_A$ or equivalent to a nice word with respect to $\rho$.
	\end{lemma}
	
	\begin{proof}
		Let $w \in W_{A \cup \simpleset{x}}$.
		If the set
		$$
			\set{w' \in [vu]_{\sim_\rho}}{w \text{ is split as } w = uv \text{ for some } u,v \in W_{A \cup \simpleset{x}}}
		$$
		contains a word $w' \in W_A$ we are done.
		Otherwise, choose $w'$ from it of minimal length and such that $q$ is minimal where $w' = px^{\pm1}q$ and $p \in W_{A \cup \simpleset{x}}$, $q \in W_A$.
		Let $w = uv$ be the witnessing split for $w' \in [vu]_{\sim_\rho}$.
		In fact, this implies that $q$ is the empty word, for if otherwise we may adjust the split $w = uv$ to move the $q$ to the other side as this does not increase the length of $w'$.
		
		If $w' = x^{\pm n}$ for some $n > \omega$ we are done.
		Otherwise let $w' = x^{k}qx^{\pm1}$ where $q \in W_{A \cup \simpleset{x}}$ and $k \in \ZZ$ with $|k|$ maximal.
		We may assume that $k = 0$, for if otherwise we may adjust the split $w = uv$ to move the $x^k$ to the other side as this does not increase the length of $w'$ and does not introduce a non-empty $q$ as above.
		
		Finally, we may choose $l < \omega$, $u_i \in W_{A}$ and $k_i \in \ZZ \setminus \simpleset{0}$
		$$
			w' = u_l x^{k_l} u_{l-1} x^{k_{l-1}} \dots u_1 x^{k_1} u_0x^{k_0}.
		$$
		In fact, $u_i \notin \ker(\rho)$, for if otherwise $w'$ is equivalent to a shorter word, contradicting its minimality.
		Thus, $w'$ is nice.
	\end{proof}
	
	\begin{corollary}\label{cor_reduce_to_nice_words}
		Let $f \in \Sinf$ and assume for all nice words $w \in W_{\rho,x}$ we have $\fix(\rho[f](w))$ is finite.
		Then $\rho[f]$ induces a cofinitary representation.
	\end{corollary}
	
	\begin{proof}
		Let $w \in W_{A \cup \simpleset{x}}$.
		By the previous lemma write $w = uv$ where $vu$ is equivalent to a word in $W_A$ or equivalent to a nice word $w'$ with respect to $\rho$.
		Then $\rho[f](vu) = \rho[f](w')$, so that Proposition \ref{prop_same_number_of_fixpoints} implies
		$$
			|\fix(\rho[f](w))| = |\fix(\rho[f](uv))| = |\fix(\rho[f](vu))| = |\fix(\rho[f](w'))| < \omega,
		$$
		which proves that $\rho[f]$ is a cofinitary representation.
	\end{proof}
	
	Thus, to construct a maximal cofinitary group, we may restrict ourself to nice words.
	These have the advantage that they satisfy the following range and domain extension lemma:
	
	\begin{lemma}\label{LEM_ExtendDomain}
		Let $s \in \Sinffin$, $W_0 \subseteq W_{\rho,x}$ be a finite subset.
		Then we have
		\begin{enumerate}
			\item If $n \in \omega \setminus \dom(s)$ then for almost all $m \in \omega$ we have that $t := s \cup (n,m) \in \Sinffin$ and for every word $w \in W_0$
			$$
				\fix(\rho[s](w)) = \fix(\rho[t](w)).
			$$
			\item If $m \in \omega \setminus \ran(s)$ then for almost all $n \in \omega$ we have that $t := s \cup (n,m) \in \Sinffin$ and for every word $w \in W_0$
			$$
				\fix(\rho[s](w)) = \fix(\rho[t](w)).
			$$
		\end{enumerate}
	\end{lemma}
	
	\begin{proof}
		First, we show how $1.$ implies $2$.
		Let $m \in \omega \setminus \ran(s)$ and let $W_0^{\perp} := \set{w^\perp}{w \in W_0}$, where $w^\perp$ is constructed by replacing all occurrences of $x$ by $x^{-1}$ and vice versa.
		Note that $w$ is nice iff $w^\perp$ is nice and for any $t \in \Sinffin$ we have
		$$
			\rho[t](w) = \rho[t^{-1}](w^\perp).
		$$
		Furthermore, $m \notin \dom(s^{-1})$, so by $1.$ for almost all $n \in \omega$ we have that $t^{-1} := s^{-1} \cup (m,n) \in \Sinffin$ and for every word $w^\perp \in W_0^\perp$
		$$
			\fix(\rho[s^{-1}](w^\perp)) = \fix(\rho[t^{-1}](w^\perp)).
		$$
		But then for every such $n < \omega$ we have $t \in \Sinffin$ and for every word $w \in W_0$
		$$
			\fix(\rho[s](w)) = \fix(\rho[s^{-1}](w^\perp)) = \fix(\rho[t^{-1}](w^\perp)) = \fix(\rho[t](w)).
		$$
		Next, we have to prove $1$.
		It suffices to prove the statement for $W_0 = \simpleset{w}$, the general case then follows iteratively.
		Consider the following cases:
		
		\textbf{Case 1:} $w = x^n$ for some $n > 0$.
		We claim that every $m \in \omega \setminus (\dom(s) \cup \ran(s) \cup \simpleset{n})$ is suitable, for if $\rho[t](w)(k)\defined$ and $\rho[s](w)(k) \undefined$ for some $k \in \dom(t)$ we can choose $i > 0$ minimal such that $\rho[s](x^i)(k) \undefined$.
		Then $\rho[s](x^{i - 1})(k) = n$, so that $\rho[t](x^i)(k) = m \notin \dom(t)$.
		But $\rho[t](w)(k) \defined$, so $i = n$ and we get $\rho[t](w)(k) = m \neq k$.
		Thus, $k \notin \fix(\rho[t](w))$.
		
		\textbf{Case 2:} $w = x^{-n}$ for some $n > 0$.
		We claim that every $m \in \omega \setminus (\dom(s) \cup \ran(s) \cup \simpleset{n})$ is suitable, for if $\rho[t](w)(k)\defined$ and $\rho[s](w)(k) \undefined$ for some $k \in \ran(t)$ we can choose $i > 0$ minimal such that $\rho[s](x^{-i})(k) \undefined$.
		Then $\rho[s](x^{-i + 1})(k) = m$, so that $i = 1$ for if otherwise $m \in \ran(\rho[s](x^{-1}))$, so we get $m \in \dom(s)$, contradicting the choice of $m$.
		But then $k = m$ and we get $\rho[t](w)(m) \neq m$ as $m \notin \dom(t)$.
		Thus, $k \notin \fix(\rho[t](w))$.
		
		For the remaining case we may choose $l < \omega$ and $u_i \in W^0_{\rho,x}$, $k_i \in \ZZ \setminus \simpleset{0}$ for $i \leq l$ such that
		$$
			w = u_l x^{k_l} u_{l-1} x^{k_{l-1}} \dots u_1 x^{k_1} u_0x^{k_0}.
		$$
		Also, since $u_i \in W^0_{\rho,x}$ we may choose $M < \omega$ large enough such that for all $i \leq l$
		\begin{enumerate}[(M1)]
			\item $\dom(s) \cup \ran(s) \cup \simpleset{n} \subseteq M$.
			\item $\rho(u_i)[\dom(s) \cup \ran(s) \cup \simpleset{n}] \subseteq M$.
			\item $\rho(u_i)^{-1}[\dom(s) \cup \ran(s) \cup \simpleset{n}] \subseteq M$.
			\item $\fix(\rho(u_i)) \subseteq M$.
		\end{enumerate}
		We will show that every $m \geq M$ is suitable, so assume $\rho[t](w)(k)\defined$ and $\rho[s](w)(k) \undefined$ for some $k \in \dom(s) \cup \ran(s) \cup \simpleset{n,m}$.
		Choose $i \leq l$ minimal and then $j \leq |k_i|$ minimal such that
		$$
			\rho[s](x^{\sign(k_i)j}u_{i-1}x^{k_{i-1}}\dots u_1x^{k_1}u_0x^{k_0})(k)\undefined.
		$$
		Then $j > 0$ by minimality of $i$.
		We consider the following two cases:
		
		\textbf{Case 1:} $k_i > 0$.
		Then
		$$
			\rho[s](x^{j-1}u_{i-1}x^{k_{i-1}}\dots u_1x^{k_1}u_0x^{k_0})(k) = n,
		$$
		so that by definition of $t$
		$$
			\rho[t](x^{j}u_{i-1}x^{k_{i-1}}\dots u_1x^{k_1}u_0x^{k_0})(k) = m.
		$$
		By (M1) $j < k_i$ contradicts $\rho[t](w)(k)\defined$, so $j = k_i$ and we get
		$$
			\rho[t](x^{k_i}u_{i-1}x^{k_{i-1}}\dots u_1x^{k_1}u_0x^{k_0})(k) = m
		$$
		By (M3) and (M4) we get
		$$
			\rho[t](u_ix^{k_i}u_{i-1}x^{k_{i-1}}\dots u_1x^{k_1}u_0x^{k_0})(k) \notin \dom(s) \cup \ran(s) \cup \simpleset{n,m},
		$$
		so $i < l$ contradicts $\rho[t](w)(k) \defined$.
		Thus $i = l$ and
		$$
			\rho[t](w)(k) \notin \dom(s) \cup \ran(s) \cup \simpleset{n,m},
		$$
		which proves that $k \notin \fix(\rho[t](w))$.
		
		\textbf{Case 2:} $k_i < 0$.
		Then
		$$
			\rho[s](x^{-j+1}u_{i-1}x^{k_{i-1}}\dots u_1x^{k_1}u_0x^{k_0})(k) = m,
		$$
		so that by definition of $t$
		$$
			\rho[t](x^{-j}u_{i-1}x^{k_{i-1}}\dots u_1x^{k_1}u_0x^{k_0})(k) = n.
		$$
		By (M1) we have $j = 1$ and we get
		$$
			\rho[s](u_{i-1}x^{k_{i-1}}\dots u_1x^{k_1}u_0x^{k_0})(k) = m.
		$$
		If $i > 0$ by minimality of $i$ we would have
		$$
			\rho[s](x^{k_{i-1}}\dots u_1x^{k_1}u_0x^{k_0})(k) \in \dom(s) \cup \ran(s) \cup \simpleset{n},
		$$
		which contradicts (M2).
		Thus, $i = 0$, i.e.\ $k = \rho[s](\epsilon)(k) = m$.
		Further, we have
		$$
			\rho[t](x^{k_l}u_{l-1}x^{k_{l-1}}\dots u_1x^{k_1}u_0x^{k_0})(m) \in \dom(s) \cup \ran(s) \cup \simpleset{n,m}.
		$$
		But then (M2) and (M4) imply
		$$
			\rho[t](w)(m) \neq m
		$$
		which proves that $k \notin \fix(\rho[t](w))$.
	\end{proof}

	\subsection{Elimination of intruders for maximal cofinitary groups} \label{SECSUB_ELIMINATING_MCG}
	
	Finally, we will prove $\eoi(\aG)$, so that Theorem \ref{THM_SacksUniversal} will give us the following new result:
	
	\begin{corollary}\label{COR_MCG}
		Under $\sf{CH}$ there is a universally Sacks-indestructible maximal cofinitary group.
	\end{corollary}
	
	The proof will follow a similar structure as the proof of Lemma \ref{LEM_EILforED} for med families, while using the case distinctive argumentation as in the extending domain/range Lemma \ref{LEM_ExtendDomain}.
	We will need to handle $W^1_{\rho,x}$, i.e.\ nice words with exactly one occurrence of $x$, differently, so let us denote $W^{>1}_{\rho, x} := \bigcup_{n > 1} W^n_{\rho,x}$.

	\begin{lemma}\label{LEM_MCG}
		Let $A$ be countable and assume $\rho:A \to \Sinf$ induces a cofinitary representation.
		Let $p \in \SS^{\aleph_0}$ and $\dot{g}$ be a name for an element of $\Sinf$ such that
		$$
			p \forces \dot{g} \notin V \text{ and } \rho[\dot{g}] \text{ induces a cofinitary representation}.
		$$
		Then there is $q \extends p$ and $f \in \Sinf$ such that $\rho[f]$ induces a cofinitary representation and
		\[
			q \forces f \text{ and } \dot{g} \text{ are not eventually different}.
		\]
	\end{lemma}

	\begin{proof}
		First, we prove that for all $w \in W^1_{\rho,x}$
		$$
			p \forces \fix(\rho[\dot{g}](w)) \text{ is finite}.
		$$
		Otherwise, choose $w \in W^1_{\rho, x}$ and $q \leq p$ such that
		$$
			q \forces \fix(\rho[\dot{g}](w)) \text{ is infinite}.
		$$
		Then by assumption on $\dot{g}$ we get
		$$
			q \forces \rho[\dot{g}](w) = \id.
		$$
		Write $w := ux^{\pm1}$, where $u\in W_A$.
		Then
		$$
			q \forces \dot{g}^{\pm1} = \rho[{\dot{g}]}(x^{\pm1}) = \rho(u)^{-1},
		$$
		so depending on the case we get
		$$
			q \forces \dot{g} = \rho(u)^{-1} \quad \text{ or } \quad q \forces \dot{g} = \rho(u).
		$$
		Either way we obtain $q \forces ``\dot{g} \in V"$, a contradiction.
		Let $\seq{w_n}{n < \omega}$ enumerate all nice words, so that $w_m$ is a subword of $w_n$ implies that $m \leq n$.
		We define a fusion sequence $\seq{p_n}{n < \omega}$ below $q_0$, a sequence $\seq{K_n < \omega}{n < \omega}$ and an increasing sequence $\seq{f_n \in \Sinffin}{n < \omega}$ with the following properties:
		\begin{enumerate}
			\item $n \in \ran(f_n) \cap \dom(f_n)$.
			\item If $w_n \in W^1_{\rho,x}$ then $p_{n + 1} \forces \fix(\rho[\dot{g}](w_n)) \subseteq K_n$
			\item For all $m \leq n$ we have $\fix(\rho[f_n](w_m)) \subseteq K_m$.
			\item $p_{n + 1} \forces \exists l \in (\dom(f_n) \setminus \bigcup_{m < n} \dom(f_m)) \ f_n(l) = \dot{g}(l)$.
		\end{enumerate}
		To see that this proves the theorem, let $q$ be the fusion of $\seq{p_n}{n < \omega}$ and define $f := \bigcup_{n < \omega} f_n$.
		By $(1)$ and since $\seq{f_n}{n < \omega}$ is an increasing sequence of partial injections we have that $f \in \Sinf$.
		By $(3)$ we have that $\rho[f]$ induces a cofinitary representation, since for every $m < \omega$ we have that $\fix(\rho[f](w_m)) \subseteq K_m$.
		Finally, $(4)$ implies that
		$$
			q \forces f \text{ and } \dot{g} \text{ are not eventually different}.
		$$
		So set $p_0 := p$ and assume that $q_n$ is defined - we then have to construct $p_{n + 1}, K_n$ and $f_n$.
		If $n \notin \bigcup_{m < n}\dom(f_m)$ or $n \notin \bigcup_{m < n} \ran(f_m)$ use Lemma \ref{LEM_ExtendDomain} to extend $f_n$ to $h_0 \in \Sinffin$ with $n \in \dom(h_0) \cap \ran(h_0)$ while preserving $(3)$.
		In case that $w_n \in W^1_{\rho,x}$ by the previous observation the set
		$$
			D_n := \set{q \leq p_n}{\exists K < \omega \ q \forces \fix(\rho[\dot{g}](w_n)) \subseteq K}
		$$
		is open dense below $p_n$.
		By Lemma \ref{LEM_ExtendDenseOpen} take $q_0 \extends_n p_n$ such that $q_0 \restr \sigma \in D_n$ for all $\sigma$ suitable for $p_n$ and $n$.
		Thus, there is a $K < \omega$ such that
		$$
			q_0 \forces \fix(\rho[\dot{g}](w_n)) \subseteq K.
		$$
		
		Next, enumerate all suitable functions $\sigma$ for $q_0$ and $n$ by $\seq{\sigma_i}{i < N}$.
		Inductively, we will define an $\leq_{n}$-decreasing sequence $\seq{q_i}{i \leq N}$ and an increasing sequence $\seq{h_i \in \Sinffin}{i \leq N}$ with
		\begin{enumerate}[(a)]
			\item For all $m < n$ we have $\fix(\rho[h_{i + 1}](w_m)) \subseteq K_m$.
			\item $q_{i + 1} \restr \sigma_i \forces \exists l \in (\dom(h_{i + 1}) \setminus \dom(h_i)) \ f_n(l) = \dot{g}(l)$.
		\end{enumerate}
		Assuming we are successful with this, we may set $f_{n} := h_{N}$ and choose $K' < \omega$ such that $\fix(\rho[f_{n}](w_n)) \subseteq K'$ and define
		$$
			K_n := \max (K, K')
		$$
		Then, we took care of $(1)$ at the beginning of the construction and $(2)$ follows from the definition of~$K$.
		Furthermore, $(3)$ follows from $(a)$ for $m < n$ and by definition of $K'$ for $m = n$.
		Finally, $(4)$ follows directly from $(b)$.
		
		For the construction, let $i < N$ and assume $q_i$ and $h_i$ are defined.
		Choose $M < \omega$ large enough such that
		\begin{enumerate}[(M1)]
			\item $\ran(h_i) \cup \dom(h_i) \subseteq M$,
		\end{enumerate}
		and for all $m < n$ if we can choose $l_m < \omega$ and $u_{i,m} \in W^0_{\rho,x}$, $k_{i,m} \in \ZZ \setminus \simpleset{0}$ for $i \leq l$ such that
		$$
		w = u_l x^{k_l} u_{l-1} x^{k_{l-1}} \dots u_1 x^{k_1} u_0x^{k_0},
		$$
		then we have for all $i \leq l$ that
		\begin{enumerate}[(M1)]\setcounter{enumi}{1}
			\item $\rho(u_{i,m})[\dom(h_i) \cup \ran(h_i)] \subseteq M$,
			\item $\rho(u_{i,m})^{-1}[\dom(h_i) \cup \ran(h_i)] \subseteq M$.
			\item $\fix(\rho(u_{i,m})) \subseteq M$
		\end{enumerate}
		Finally, we also require
		\begin{enumerate}[(M1)]\setcounter{enumi}{4}
			\item for all $m < n$ we have $K_m \leq M$.
		\end{enumerate}
		Choose $r \leq q_i \restr \sigma_i$ which decides $M + 1$ values of $\dot{g}$ above $M$.
		Since, $r \forces ``\dot{g} \text{ is injective}"$ there are $n_0, m_0 \geq M$ such that $r \forces ``\dot{g}(n_0) = m_0"$.
		Set $h_{i + 1} := h_i \cup \simpleseq{n_0,m_0}$ and $q_{i + 1} \leq_{n} q_i$ so that $q_{i + 1} \restr \sigma_i = r_i$ by
		$$
			q_{i + 1}(\alpha) := 
			\begin{cases}
				r_i(\alpha) \cup \bigcup\set{q_i(\alpha)_s}{s \in \spl_n(q_i(\alpha)) \concat 2 \text{ and } s \neq \sigma_i(\alpha)} & \text{if } \alpha < n\\
				r_i(\alpha) & \text{otherwise}
			\end{cases}
		$$
		Clearly, $h_{i + 1} \in \Sinf$ by (M1) and $q_{i + 1}$ satisfies property (b) since
		$$
			q_{i + 1} \restr \sigma_i  \forces h_{i + 1}(n_0) = m_0 = \dot{g}(n_0)
		$$
		It remains to show that also property (a) is satisfied, so let $m < n$.
		First, we consider the case $w_m \in W^1_{\rho,x}$.
		But in this case we have $w_m = ux^{\pm 1}$ for $u \in W^0_{\rho,x} \cup \simpleset{\epsilon}$, so that for every $l \in \dom(h_i) \cup \ran(h_i)$ we have 
		$$
			\rho[h_{i + 1}](w_m)(l) = \rho[h_i](w_m)(l),
		$$
		and for $l \in \simpleset{n_0, m_0}$ we may apply $(2)$ inductively to obtain
		$$
			q_i \restr \sigma_i \forces \rho[h_{i + 1}](w_m)(l) = \rho[\dot{g}](w_m)(l) \text{ and } l \notin \fix(\rho[\dot{g}](w_m)),
		$$
		since $n_0, m_0 \geq K_m$.
		Thus, $\fix(\rho[h_{i + 1}](w_m)) = \fix(\rho[h_i](w_m)) \subseteq K_m$ inductively by (a).
		
		Next, we consider the case $w_m \in W^{>1}_{\rho,x}$.
		Again, for all $l \in \dom(h_i) \cup \ran(h_i)$ we have
		$$
			\rho[h_{i + 1}](w_m)(l) = \rho[h_i](w_m)(l)
		$$
		by properties (M1), (M2) and (M3).
		Thus, it remains to consider the cases $l \in \simpleset{n_0, m_0}$.
		We show that $\rho[h_{i + 1}](w_m)(l) \undefined$ which finishes the proof.
		If $l = n_0$ we may write $w_m = vx^{\pm1} u x$ for $v \in W_{A \cup \simpleset{x}}$ and $u \in W^0_{\rho,x} \cup \simpleset{\epsilon}$.
		By (M2) or (M3) we have that
		$$
			\rho[h_{i + 1}](ux)(n_0) \notin \dom(h_{i}) \cup \ran(h_i).
		$$
		Further, if $w_m = v x^{-1} u x$ we have $u \in W^0_{\rho,x}$, so that by (M4) $\rho(u)(m_0) \neq m_0$.
		Hence,
		$$
			\rho[h_{i + 1}](ux)(n_0) \neq m_0.
		$$
		Otherwise, $w_m = vx u x$ and $ux \in W^1_{\rho,x}$ is a subword of $w_m$, so choose $m' < m$ with $w_{m'} = ux$, so by (M5) and (2) we get
		$$
			\rho[h_{i + 1}](ux)(n_0) \neq n_0.
		$$
		Thus, in both cases $\rho[h_{i + 1}](x^{\pm 1} u x) \undefined$.
		
		Finally, for $l = m_0$ we may write $w_m = vx^{\pm1} u x^{-1}$ for $v \in W_{A \cup \simpleset{x}}$ and $u \in W^0_{\rho,x} \cup \simpleset{\epsilon}$.
		By (M2) or (M3) we have that
		$$
			\rho[h_{i + 1}](ux^{-1})(m_0) \notin \dom(h_{i}) \cup \ran(h_i).
		$$
		Further, if $w_m = v x u x^{-1}$ we have $u \in W^0_{\rho,x}$, so that by (M4) $\rho(u)(n_0) \neq n_0$.
		Hence,
		$$
			\rho[h_{i + 1}](ux^{-1})(m_0) \neq n_0.
		$$
		Otherwise, $w_m = vx u x^{-1}$ and $ux^{-1} \in W^1_{\rho,x}$ is a subword of $w_m$, so again choose $m' < m$ with $w_{m'} = ux^{-1}$, so by (M5) and (2) we get
		$$
			\rho[h_{i + 1}](ux^{-1})(m_0) \neq m_0.
		$$
		Thus, in both cases $\rho[h_{i + 1}](x^{\pm 1} u x^{-1}) \undefined$.
	\end{proof}
	
	\subsection{Independent families and ultrafilters}\label{SECSUB_UANDI}
	
	In this section we want to consider independent families and ultrafilters.
	We show that both of these types of families are arithmetically definable, so that we may apply Theorem~\ref{THM_MainTheorem}.
	Further, we will prove $\eoi(\u)$, so by Theorem~\ref{THM_SacksUniversal} we may easily construct a universally Sacks-indestructible ultrafilter under $\sf{CH}$ without referring to some kind of selectivity.
	At this time, we dot not know if the same is possible for independent families.
	Remember the following two definitions:
	
	\begin{definition}
		Let $\A$ be a subset of $\infsubset{\omega}$.
		Denote with $\text{FF}(\A)$ the set of all finite partial functions $f:\A \to 2$.
		Given $f \in \text{FF}(A)$ we define
		$$
			\A^f := (\bigcap_{A \in f^{-1}[0]}A) \cap (\bigcap_{A \in f^{-1}[1]} A^c).
		$$
		The family $\A$ is independent iff for all $f \in \text{FF}(A)$ we have that $\A^f$ is infinite.
		$\A$ is called maximal iff it is maximal {w.r.t.} inclusion.
		The corresponding cardinal characteristic is the independence number $\i$:
		\[
			\i := \min{\set{\left|\A\right|}{\A \text{ is a maximal independent family}}}.
		\]
	\end{definition}
	
	\begin{proposition}
		Maximal independent family are an arithmetical type.
	\end{proposition}

	\begin{proof}
		Using the same coding of $\infsubset{\omega}$ by reals as for mad families let $\psi_0(w_0)$ express \textquoteleft $w_0$ codes an infinite subset of $\omega$'.
		For $n > 0$ there is an arithmetical formula $\psi_n(w_0, \dots, w_n)$ expressing
		\begin{align*}
			&(\ran(w_0) \cap \ran(w_1) \cap \dots \cap \ran(w_n) \text{ is infinite})\\
			\text{ and } &(\ran(w_0) \cap \ran(w_1) \cap \dots \cap \ran(w_n)^c \text{ is infinite})\\
			\text{ and } &\dots\\
			\text{ and } &(\ran(w_0)^c \cap \ran(w_1)^c \cap \dots \cap \ran(w_n)^c \text{ is infinite}).
		\end{align*}
		Analogously, we let $\chi_0(v)$ express \textquoteleft$v$ codes an infinite subset of $\omega$' and for $n > 0$ choose $\chi_n(v, w_1, \dots, w_n)$ expressing
		\begin{align*}
			&(\ran(v) \cap \ran(w_1) \cap \dots \cap \ran(w_n) \text{ is infinite})\\
			\text{ and } &(\ran(v) \cap \ran(w_1) \cap \dots \cap \ran(w_n)^c \text{ is infinite})\\
			\text{ and } &\dots\\
			\text{ and } &(\ran(v)^c \cap \ran(w_1)^c \cap \dots \cap \ran(w_n)^c \text{ is infinite}).
	\end{align*}
		Thus, maximal independent families are an arithmetical type.
	\end{proof}

	\begin{definition}
		We say a subset $\A \subseteq \infsubset{\omega}$ satisfies the strong finite intersection property (SFIP) iff $\bigcap \A_0$ is infinite for all $\A_0 \in \finsubset{\A}$.
		In this case we define the generated filter of $\A$ by
		$$
			\simpleseq{\A} := \set{C \subseteq \omega}{\exists \A_0 \in \finsubset{\A} \ \bigcap \A_0 \subseteq C}.
		$$
		We call $\A$ an ultrafilter subbasis iff $\simpleseq{\A}$ is an ultrafilter.
		The corresponding cardinal characteristic is the ultrafilter number $\u$:
		\[
			\u := \min{\set{\left|\A\right|}{\A \text{ is an ultrafilter subbasis}}}.
		\]
	\end{definition}
	
	\begin{proposition}
		Ultrafilter subbases are an arithmetical type.
	\end{proposition}

	\begin{proof}
		Using the same coding of $\infsubset{\omega}$ by reals as for mad families let $\psi_0(w_0)$ express \textquoteleft $w_0$ codes an infinite subset of $\omega$'.
		For $n > 0$ there is an arithmetical formula $\psi_n(w_0, \dots, w_n)$ expressing
		$$
			\bigcap_{i = 0}^n \ran(w_i) \text{ is infinite}.
		$$
		Analogously, let $\chi_0(v)$ express \textquoteleft$v$ codes an infinite subset of $\omega$'.
		Furthermore, for $n > 0$ there is an arithmetical formula $\chi_n(v, w_1, \dots, w_n)$ expressing
		$$
			\bigcap_{i = 1}^n \ran(w_i) \not\subseteq \ran(v) \text{ and } \bigcap_{i = 1}^n \ran(w_i) \not\subseteq \ran(v)^c
		$$
		Thus, ultrafilter subbases are an arithmetical type.
	\end{proof}
	
	Applying, Theorem~\ref{THM_MainTheorem} to these two cases yields:
	
	\begin{corollary}\label{COR_UANDI}
		Every $\SS^{\aleph_0}$-indestructible independent family and every $\SS^{\aleph_0}$-indestructible ultrafilter is universally Sacks-indestructible.
	\end{corollary}
	
	In \cite{ShelahUlessI} Shelah implicitly constructed a universally Sacks-indestructible maximal independent family, whereas in \cite{Laver} Laver proved that every selective ultrafilter is preserved by any product of Sacks-forcing.
	Thus, together with the corollary above we obtain another proof that every selective ultrafilter is universally Sacks-indestructible.

	Hence, universally Sacks-indestructible independent families and ultrafilters may exist, but both of these construction refer to some kind of selectivity.
	Finally, we show that the colouring principle ${\sf HL}_\omega$ proven by Laver in \cite{Laver} can be used to prove $\eoi(\u)$.
	Hence, we may also directly construct a universally Sacks-indestructible ultrafilter under ${\sf CH}$ without using selectivity.
	We use the following notion introduced by Laver in \cite{Laver}.
	
	\begin{definition}
		For $p \in \SS$ let $p^{(n)}$ be the $n$-th level of $p$.
		For $A \subseteq \omega$ and $p \in \SS^{\aleph_0}$ let
		$$
		{\bigotimes}^A p := \bigcup_{n \in A} \bigotimes_{i < \omega} p(i)^{(n)},
		$$
		where $\bigotimes_{i < \omega} p(i)^{(n)}$ is the cartesian product.
	\end{definition}
	
	\begin{lemma}[$\text{HL}_\omega$, Laver, \cite{Laver}]
		Let $p \in \SS^{\aleph_0}$ and $c: {\bigotimes}^\omega p \to 2$.
		Then there is $q \extends p$ and $A \subseteq \omega$ such that $c \restr {\bigotimes}^Aq$ is monochromatic.
	\end{lemma}
	
	The following strengthening is implicit in \cite{Laver}, but for completeness we provide a proof:
	
	\begin{corollary}
		Let $p \in \SS^{\aleph_0}$, $A \subseteq \omega$, $k < \omega$ and $c: {\bigotimes}^A p \to k$.
		Then there is $q \extends p$ and $B \subseteq A$ such that $c \restr {\bigotimes}^Bq$ is monochromatic.
	\end{corollary}
	
	\begin{proof}
		First, we prove the statement for $k = 2$, so let $p \in \SS^{\aleph_0}$, $A \subseteq \omega$ and $c: {\bigotimes}^Ap \to 2$.
		Let $\seq{a_n}{n < \omega}$ be the increasing enumeration of $A$.
		For every $i < \omega$ we may choose $p'(i) \in \SS$ such that $p'(i) \extends p(i)$ and $\left|p(i)^{(a_n)}\right| \leq 2^n$ for all $n < \omega$.
		Then, we may choose an $\trianglelefteq$-order-preserving map $\Phi_i:p'(i) \to \fincantorspace$ with $\Phi_i[p'(i)^{(a_n)}] \subseteq (\fincantorspace)^{(n)}$ and $\Phi_i$ is injective on $p'(i)^{(a_n)}$ for all $n < \omega$.
		Then, $\ran(\Phi_i) \in \SS$, so define $r \in \SS^{\aleph_0}$ by $r(i) := \ran(\Phi_i)$ for $i < \omega$.
		Next, define $d:{\bigotimes}^\omega r \to 2$ by
		$$
		d(s) := c(t), \text{ where } t \text{ is the unique element of } {\bigotimes}^Ap \text{ with } \Phi_i(t(i)) := s(i).
		$$
		Now, by $\text{HL}_\omega$ choose $r' \extends r$ and $B \subseteq \omega$ such that $d \restr {\bigoplus}^Br'$ is monochromatic.
		Finally, for every $i < \omega$ let $q(i)$ be the preimage of $r'(i)$ under $\Phi_i$.
		Then $q(i) \in \SS$ and $q(i) \extends p'(i)$.
		Further, let $C := \set{a_n}{n \in B}$.
		Then, $q \in \SS$, $q \extends p$, $C \subseteq A$ and $c \restr {\bigoplus}^Cq$ is monochromatic.
		
		Next, we inductively prove the statement for higher $k$, so let $k + 1 > 2$, $p \in \SS^{\aleph_0}$, $A \subseteq \omega$ and $c:{\bigotimes}^Ap \to k + 1$.
		Define $c:{\bigotimes}^Ap \to k$ by
		$$
		d(s) := \max(c(s), k-1).
		$$
		By induction we may choose $q \extends p$ and $B \subseteq A$ such that $d \restr {\bigotimes}^Bq$ is monochromatic.
		In case that $d \restr {\bigotimes}^Bq = l$ for some $l < k$ we have $c \restr {\bigotimes}^Bq = d \restr {\bigotimes}^Bq$ is monochromatic.
		Otherwise, we have $c \restr {\bigotimes}^Bq \to \simpleset{k, k-1}$, so by case $k = 2$ we may choose $C \subseteq B$ and $r \extends q$ such that $c \restr {\bigotimes}^Cr$ is monochromatic.
	\end{proof}
	
	\begin{lemma}
		For every $A \in \infsubset{\omega}$ and $\SS^{\aleph_0}$-name $\dot{B}$ for a subset of $\omega$ we have
		$$
		\SS^{\aleph_0} \forces \exists B \in (\infsubset{A})^V \text{ with } B \subseteq \dot{B} \text{ or } B \subseteq \dot{B}^c
		$$
	\end{lemma}
	
	\begin{proof}
		Let $\seq{a_n}{n < \omega}$ enumerate $A$.
		Using a fusion construction we may choose an increasing sequence $C = \set{c_n}{n < \omega}$ and $q \extends p$ such that for every $s \in \bigotimes_{i < \omega} q(i)^{(c_n)}$ we have that $q_s$ decides \textquoteleft$a_n \in \dot{B}$\textquoteright.
		Define a colouring $c:{\bigotimes}^Cq \to 2$ for $s \in \bigotimes_{i < \omega} q(i)^{(c_n)}$ by
		$$
		c(s) =
		\begin{cases}
			0 & \text{ if } q_s \forces a_n \in \dot{B},\\
			1 & \text{ if } q_s \forces a_n \notin \dot{B}.
		\end{cases}
		$$
		By the previous corollary choose $D \subseteq C$ and $r \extends q$ such that $c \restr {\bigotimes}^Dr$ is monochromatic and let $B := \set{a_n}{c_n \in D}$, so that $B \in \infsubset{A}$.
		Finally, if $c \restr {\bigotimes}^Dr \equiv 0$, then
		$$
		q \forces B \subseteq \dot{B}
		$$
		whereas if $c \restr {\bigotimes}^Dr \equiv 1$, then
		$$
		q \forces B \subseteq \dot{B}^c
		$$
		completes the proof.
	\end{proof}
	
	Notice that for $A = \omega$ the previous corollary precisely states that $\SS^{\aleph_0}$ preserves $\infsubset{\omega}$ as an unsplit/reaping family.
	Finally, we are in position to prove $\eoi(\u)$:
	
	\begin{corollary}
		Let $p \in \SS^{\aleph_0}$, assume $\A$ is countable and satisfies the SFIP and let $\dot{B}$ be a $\SS^{\aleph_0}$-name for a subset of $\omega$.
		Then there is $q \extends p$ and $B \subseteq \omega$ such that $\A \cup \simpleset{B}$ satisfies the SFIP and
		$$
		q \forces B \subseteq \dot{B} \text{ or } B \subseteq \dot{B}^c
		$$
	\end{corollary}
	
	\begin{proof}
		By assumption on $\A$ we may choose an infinite pseudo-intersection $A$ of $\A$.
		By the previous lemma choose $q \leq p$ and $B \subseteq A$ such that
		$$
		q \forces B \subseteq \dot{B} \text{ or } B \subseteq \dot{B}^c
		$$
		But $A$ is a pseudo-intersection of $\A$ and $B \subseteq A$, so that $\A \cup \simpleset{B}$ satisfies the SFIP.
	\end{proof}
	
	Hence, Theorem \ref{THM_SacksUniversal} yields another proof of the following result.
	As discussed before, selective ultrafilters may be used instead to obtain the following result:
	
	\begin{corollary}\label{COR_MCG}
		Under $\sf{CH}$ there is a universally Sacks-indestructible ultrafilter.
	\end{corollary}
	
	\begin{question}
		Does $\eoi(\i)$ hold?
	\end{question}
	
	\subsection{Bounding, dominating and other types of families}\label{SECSUB_OTHER}
	
	Finally, in this last section we take a look at a few other types of families which do not have any restrictions as to what constitutes a family of that type as they do not have additional structure, but only the notion of an intruder really matters.
	
	\begin{definition}
		Given $f, g \in \bairespace$ we write $f \dominatedby g$ iff $f(n) < g(n)$ for all but finitely many $n < \omega$.
		A family $\B \subseteq \bairespace$ is called unbounded iff for all $g \in \bairespace$ there is $f \in \B$ such that $f \ndominatedby g$.
		The corresponding cardinal characteristic is the (un-)bounding number $\b$:
		$$
			\b := \min{\set{\left|\B\right|}{\B \text{ is an unbounded family}}}.
		$$
		A family $\D \subseteq \bairespace$ is called dominating iff for all $g \in \bairespace$ there is $f \in \D$ such that $g \dominatedby f$.
		The corresponding cardinal characteristic is the dominating number $\d$:
		\[
			\d := \min{\set{\left|\D\right|}{\D \text{ is a dominating family}}}.
		\]
	\end{definition}
	
	We verify that also these kinds of families may be considered in our framework:
	
	\begin{proposition}
		Unbounded and dominating families are arithmetical types.
	\end{proposition}

	\begin{proof}
		Set $\psi_n :\equiv \top$ for all $n < \omega$. 
		Analogously, set $\chi_n :\equiv \top$ for all $n \neq 1$.
		Finally, define the formula $\chi_1(v,w_1)$ to be
		$$
			\exists N \forall n (n > N \text{ implies } w_1(n) < v(n)),
		$$
		expressing \textquoteleft $w_1 \dominatedby v$'.
		Thus, unbounded families are an arithmetical type.
		Dominating families can be defined analogously.
	\end{proof}
	
	\begin{definition}
		Given $A,S \in \infsubset{\omega}$ we say $S$ splits $A$ iff $A \cap S$ and $A \cap S^c$ are infinite.
		A family $\S \subseteq \bairespace$ is called splitting iff for all $A \in \infsubset{\omega}$ there is $S \in \S$ such that $S$ splits $A$.
		The corresponding cardinal characteristic is the splitting number $\s$:
		$$
			\s := \min{\set{\left|\S\right|}{\S \text{ is a splitting family}}}.
		$$
		A family $R \subseteq \bairespace$ is called reaping or unsplit iff for all $S \in \infsubset{\omega}$ there is $R \in \R$ such that $S$ does not split $R$.
		The corresponding cardinal characteristic is the reaping number $\r$:
		$$
			\r := \min{\set{\left|\R\right|}{\R \text{ is a reaping family}}}.
		$$
	\end{definition}
	
	\begin{proposition}
		Reaping and splitting families are arithmetical types.
	\end{proposition}

	\begin{proof}
		Using the same coding of $\infsubset{\omega}$ by reals as for mad families let $\psi_0(w_0)$ express \textquoteleft $w_0$ codes an infinite subset of $\omega$' and set $\psi_n :\equiv \top$ for $n > 0$.
		Analogously, we let $\chi_0(v)$ express \textquoteleft$v$ codes an infinite subset of $\omega$'.
		Further, there is an arithmetical formula $\chi_1(v,w_1)$ expressing
		$$
			\ran(v) \cap \ran(w_1) \text{ is finite or } \ran(v) \cap \ran(w_1)^c \text{ is finite}
		$$
		and set $\chi_n :\equiv \ \top$ for all $n > 1$.
		Thus, splitting families are an arithmetical type.
		Reaping families can be defined analogously.
	\end{proof}
	
	\begin{corollary}\label{COR_OTHER}
		Every $\SS^{\aleph_0}$-indestructible unbounded/dominating/reaping/splitting family is universally Sacks-indestructible.
	\end{corollary}

	For example one special case is that $\bairespace$-bounding is equivalent to $\bairespace$ being preserved as a dominating family.
	In this special case proving that $\SS^{\aleph_0}$ is $\bairespace$-bounding yields another way to show that any countably supported product or iteration of Sacks-forcing of any length is $\bairespace$-bounding.
	
	It is easy to see that $\SS^{\aleph_0}$ preserves $\infsubset{\omega}$ as a splitting family and in \cite{Laver} Laver proved that $\SS^{\aleph_0}$ preserves $\infsubset{\omega}$ as a reaping family.
	Thus, in these special cases we obtain another proof that any countably supported product or iteration of Sacks-forcing of any length preserves $\infsubset{\omega}$ as a splitting and reaping family.
	
	The list of families to which we can apply Theorem~\ref{THM_MainTheorem} does not end here though.
	For example we could also apply the same arguments to evading and predicting families or witnesses for $\p$ and other types of families.
	
	\bibliographystyle{plain}
	\bibliography{refs}
	
\end{document}